\documentclass[12pt]{amsart}

\usepackage{amssymb, amsthm, amsmath}

\usepackage{gksymb}

\usepackage[single]{accents}

\newtheorem{theorem}{Theorem}[section]
\newtheorem{lemma}[theorem]{Lemma}
\newtheorem{prop}[theorem]{Proposition}
\newtheorem{corollary}[theorem]{Corollary}

\newtheorem{question}[theorem]{Question}

\theoremstyle{definition}
\newtheorem{definition}[theorem]{Definition}
\newtheorem{notation}[theorem]{Notation}
\newtheorem{example}[theorem]{Example}

\newtheorem{claim}{Claim}

\theoremstyle{remark}
\newtheorem{remark}[theorem]{Remark}

\numberwithin{equation}{section}

\newcommand{\refl}[1]{\accentset{\leftarrow}{#1}}

\newcommand{\ip}[2]{\langle #1, #2 \rangle}
\newcommand{\gkmeas}{\frac{dz}{z}\frac{dw}{w}}

\newcommand{\EOEx}{\hspace*{0pt} \hfill $\triangleleft$}
\renewcommand{\vec}[1]{\mathbf{#1}}
\newcommand{\mcc}{\raise 0.5ex \hbox{c}}
\newcommand{\cbidisk}{\overline{\mathbb{D}^2}}
\newcommand{\aocond}{\gkboxurperpdn_\mu = \gkboxrperpdn_\mu}
\newcommand{\bLam}{\mathbf{\Lambda}}

\title{Polynomials with no zeros on the bidisk}
\author{Greg Knese} 
\address{University of California, Irvine, Irvine, CA 92697-3875}
\date{\today} 
\email{gknese@uci.edu}


\begin{document}
\bibliographystyle{plain}

\maketitle

\begin{abstract} We prove a detailed sums of squares formula for
  two variable polynomials with no zeros on the bidisk $\mathbb{D}^2$
  extending previous versions of such a formula due to Cole-Wermer and
  Geronimo-Woerdeman.  The formula is related to the
  Christoffel-Darboux formula for orthogonal polynomials on the unit
  circle, but the extension to two variables involves issues of
  uniqueness in the formula and the study of ideals of two variable
  orthogonal polynomials with respect to a positive Borel measure on
  the torus which may have infinite mass.  We present applications to
  two variable Fej\'er-Riesz factorizations, analytic extension
  theorems for a class of bordered curves called distinguished
  varieties, and Pick interpolation on the bidisk.
\end{abstract}

\section{Introduction}
Let $q \in \mathbb{C}[z,w]$ be a polynomial of degree $(n,m)$ (degree
$n$ in $z$ and degree $m$ in $w$).  Suppose $q$ has no zeros on the
unit bidisk $\mathbb{D}^2 := \mathbb{D} \times \mathbb{D} \subset
\mathbb{C}^2$. Then, $q$ satisfies the following ``sums of (Hermitian)
squares'' formula: there exist polynomials $A_j \in \mathbb{C}[z,w]$,
for $j=1,\dots, n$, and $B_k \in \mathbb{C}[z,w]$, for $k=1,\dots, m$
such that
\begin{equation} \label{1stsos}
|q(z,w)|^2 - |\refl{q}(z,w)|^2 = (1-|z|^2)\sum_{j=1}^n |A_j(z,w)|^2 +
(1-|w|^2) \sum_{k=1}^m |B_k (z,w)|^2
\end{equation}
where $\refl{q}$ is the ``reflection'' of $q$:
\[
\refl{q}(z,w) = z^n w^m \overline{q\left(\frac{1}{\bar{z}},
    \frac{1}{\bar{w}}\right)}.
\]
This was first proved in Cole-Wermer \cite{CW99}.  Here is an example.

\begin{example} \label{easyexample1}
The polynomial $q(z,w) = 2-z-w$ has degree $(1,1)$ and no zeros on
$\mathbb{D}^2$. The reflection of $q$ is $\refl{q}(z,w) =
2zw-w-z$. The sum of squares decomposition for $q$ is rather simple:
\[
|2-z-w|^2 - |2zw - w - z|^2 = (1-|z|^2)2|1-w|^2 + (1-|w|^2)2|1-z|^2.
\]
\EOEx
\end{example}

There are several reasons why we deem this formula interesting. First,
it can be used to give direct proofs of And\^{o}'s inequality from
operator theory (in Cole-Wermer \cite{CW99}) and Agler's Pick
interpolation theorem for the bidisk.  Second, \eqref{1stsos} can be
thought of as a two variable version of the Christoffel-Darboux
formula for orthogonal polynomials on the unit circle. The
Christoffel-Darboux formula has great importance in the theory of
orthogonal polynomials on the unit circle as evidenced by its
prominence in the book Simon \cite{bS05} and its featured role in the
survey Simon \cite{bS08}.  Third, the most obvious analogue of
\eqref{1stsos} in three or more variables is false as it would imply a
three operator version of And\^{o}'s inequality (something known to be
false). Fourth, \eqref{1stsos} can be used to prove a determinantal
representation for a class of algebraic curves in $\mathbb{C}^2$
called distinguished varieties.

One drawback to the Cole-Wermer formula is that the sums of squares
decomposition is not unique.  Consider the following example.

\begin{example} \label{hardexample}
  Let $f(z,w) = 2-zw-z^2w = q(zw,z^2w)$, where $q$ is from the
  previous example.  Again, $f$ has no zeros on $\mathbb{D}^2$ and we
  define  
\[
\refl{f}(z,w) = z^3 w^2 \overline{f(1/\bar{z}, 1/\bar{w})} =
\refl{q}(zw, z^2w)
\]
In this case, if we replace $z$ with $zw$ and $w$ with $z^2w$
in the sums of squares decomposition for $q$ we get
\[
\begin{aligned}
  |2-zw - z^2w|^2 &- |2z^3w^2-z^2w - zw|^2 \\
  & = (1-|zw|^2)2|1-z^2w|^2 + (1-|z^2w|^2)2|1-zw|^2
\end{aligned}
\]
and there are a number of ways to decompose this further.  One way is
\[
\begin{aligned}
|2-zw-z^2w|^2 &- |2z^3w^2 - z^2w - zw|^2 \\
&= (1-|z|^2)|\mathbf{E}(z,w)|^2
+ (1-|w|^2)|\mathbf{F}(z,w)|^2
\end{aligned}
\]
with
\[
\mathbf{E}(z,w) = \sqrt{2} \begin{pmatrix} 1-z^2w \\ w-zw^2
  \\zw-z^2w^2 \end{pmatrix} = \sqrt{2} \begin{pmatrix} 1 & 0 & -w \\ w
  & -w^2 & 0 \\ 0 & w & -w^2 \end{pmatrix} \begin{pmatrix} 1 \\ z \\
  z^2 \end{pmatrix}
\]
\[
\mathbf{F}(z,w) = \sqrt{2}\begin{pmatrix} z-z^3w \\
  1-zw \end{pmatrix} = \sqrt{2} \begin{pmatrix} z & -z^3 \\ 1 &
  -z \end{pmatrix} \begin{pmatrix} 1 \\ w \end{pmatrix}.
\]

Another way is
\[
\mathbf{E}(z,w) = \begin{pmatrix} \sqrt{2}(z-z^2w) \\ z-z^2 \\
  2-zw-z^2w \end{pmatrix}  = \begin{pmatrix} 0 & \sqrt{2} & -\sqrt{2}w
  \\ 0 & 1 & -1 \\ 2 & -w & -w \end{pmatrix} \begin{pmatrix} 1 \\ z \\
  z^2 \end{pmatrix} \text{ and}
\]
\[
\mathbf{F}(z,w) = \begin{pmatrix} z +z^2 -2z^3w \\
  z^2-z^3 \end{pmatrix} = \begin{pmatrix} z + z^2 & -2z^3 \\ z^2-z^3 &
  0 \end{pmatrix} \begin{pmatrix} 1 \\ w \end{pmatrix}.
\]

These two choices for $\mathbf{E}$ and $\mathbf{F}$ are not equivalent
up to unitary multiplication  because in the first case
\[
\det \sqrt{2} \begin{pmatrix} 1 & 0 & -w \\ w
  & -w^2 & 0 \\ 0 & w & -w^2 \end{pmatrix} = 2\sqrt{2}w^3(w-1)
\]
and in the second case
\[
\det \begin{pmatrix} 0 & \sqrt{2} & -\sqrt{2}w
  \\ 0 & 1 & -1 \\ 2 & -w & -w \end{pmatrix}=2\sqrt{2}(w-1).
\]
\EOEx
\end{example}

It turns out that we can guarantee that the Cole-Wermer sums of
squares decomposition is unique if we require more.  We shall present
the main theorem after some quick notation.

\begin{notation}\label{mainthmnotation} We use $\mathbb{T}$ to denote
  the unit circle $\partial \mathbb{D}$ and $\mathbb{T}^2$ is the two
  dimensional torus, or just ``torus.'' We use $\mathbb{C}^N[z]$ to
  denote the set of $\mathbb{C}^N$ valued polynomials in the variable
  $z$; likewise, we use $\mathbb{C}^N[z,w]$ to denote the set of
  $\mathbb{C}^N$ valued polynomials in $z$ and $w$.  We define
\begin{equation} \label{lambda}
  \mathbf{\Lambda}_N(z) := \begin{pmatrix} 1 \\ z \\ \vdots \\
    z^{N-1} \end{pmatrix} \in \mathbb{C}^N[z].
\end{equation}
If $\mathbf{E}(z,w) = \sum_{j=0}^{n-1} \mathbf{E}_j(w) z^j\in
\mathbb{C}^N[z,w]$ has degree less than $n$ in $z$, we will frequently
write $\mathbf{E}$ in the matrix form
\[
\mathbf{E}(z,w) = (\mathbf{E}_0(w), \mathbf{E}_1(w), \dots,
\mathbf{E}_{n-1}(w)) \mathbf{\Lambda}_n (z) = E(w) \mathbf{\Lambda}_n(z)
\]
where $E(w) = (\mathbf{E}_0(w), \mathbf{E}_1(w), \dots,
\mathbf{E}_{n-1}(w))$ is an $N \times n$ matrix valued polynomial in
$w$.

We let $|\cdot|$ denote the standard norm on $\mathbb{C}^N$ (where the
$N$ will be understood from context) and therefore if
$\mathbf{E}=(e_1,\dots, e_N)^t \in \mathbb{C}^N[z,w]$, then
\[
|\mathbf{E}(z,w)|^2 = \sum_{j=1}^{N}
|e_j(z,w)|^2
\]
is evaluated pointwise (and does not represent any type of function
space norm).

\end{notation}

Here is an abridged version of our main theorem. We will fill in more
details in Theorem \ref{mainthmfull}.

\begin{theorem} \label{mainthm} Let $q\in \mathbb{C}[z,w]$ have degree
  at most $(n,m)$ with no zeros on $\mathbb{D}^2$ and finitely many
  zeros on $\mathbb{T}^2$.  Then, there exist vector polynomials
  $\mathbf{E} \in \mathbb{C}^n[z,w]$ and $\mathbf{F} \in
  \mathbb{C}^m[z,w]$ of degree at most $(n-1,m)$ and $(n,m-1)$
  respectively (in each component) with the property that if we write
  them in matrix form as
\[
\mathbf{E}(z,w) = E(w) \mathbf{\Lambda}_n(z) 
\qquad
\mathbf{F}(z,w) = F(z) \mathbf{\Lambda}_m(w) 
\]
where $E(w)$ is an $n\times n$ matrix polynomial of degree at most $m$
and $F(z)$ is an $m\times m$ matrix polynomial of degree at most $n$,
then
\begin{enumerate}
\item $E(w)$ is invertible for all $w \in \mathbb{D}$,
\item $z^n \overline{F(1/\bar{z})}$ is invertible for all $z \in
  \mathbb{D}$,
\item the following formula holds
  \begin{align} \label{mainsos} |q(z,w)|^2 & - |\refl{q}(z,w)|^2 \\
    \nonumber &= (1-|z|^2) |\mathbf{E}(z,w)|^2 + (1-|w|^2)
    |\mathbf{F}(z,w)|^2, \text{ and}
\end{align}

\item $\mathbf{E} \in \mathbb{C}^n[z,w]$ and $\mathbf{F}\in
  \mathbb{C}^m[z,w]$ satisfying items (1) and (3) above are unique up
  to unitary multiplication.

\end{enumerate}

\end{theorem}

A number of remarks are in order.

\begin{remark} 
As one can check, in Example \ref{hardexample}, the second choices of
$\mathbf{E}$ and $\mathbf{F}$ fit the requirements of the above
theorem, while the first choices do not.
\end{remark}

\begin{remark}
In the case of a polynomial with no zeros on the \emph{closed} bidisk
$\cbidisk$, this theorem is deducible from the work of
Geronimo-Woerdeman \cite{GW04}.  It is the goal of this paper to
extend the sums of squares decomposition with uniqueness to all
polynomials with no zeros on the \emph{open} bidisk $\mathbb{D}^2$.
Why are we concerned with such an extension?

First, it allows a direct, unified proof of the Cole-Wermer formula
which does not make use of And\^{o}'s inequality, Agler's
interpolation theorem, or any of their close relatives (the original
proof of Cole and Wermer relies heavily on these results). Going from
the case of no zeros on the closed bidisk (as in the
Geronimo-Woerdeman formula) to the general case of no zeros on the
open bidisk (as in the Cole-Wermer formula) can be accomplished with a
limiting argument (this was done in Knese \cite{gK08a}).  However,
preserving the uniqueness aspect in a limit does not seem to be
straightforward.  We will comment on this in Remark
\ref{finiteremark}.

Second, it allows us to prove a bounded analytic extension theorem
with estimates for the already alluded to curves called distinguished
varieties.  Distinguished varieties are algebraic curves in
$\mathbb{C}^2$ that exit the bidisk through the distinguished boundary
$\mathbb{T}^2$.  We proved a bounded analytic extension theorem in
Knese \cite{gK08b} using the Geronimo-Woerdeman version of the sums of
squares formula, but in that paper we were restricted to the case of
distinguished varieties with no singularities on the torus.  Theorem
\ref{mainthm} allows us to remove that restriction.

Third, our method of proof may be of interest to some as we study
orthogonal polynomials with respect to a positive Borel measure on
$\mathbb{T}^2$ which may have infinite mass.  Since such measures will
not necessarily have finite moments, methods involving doubly Toeplitz
matrices (as in Geronimo-Woerdeman \cite{GW04}) are not directly
available to us, and therefore our method of using reproducing kernels
of subspaces of polynomials from Knese \cite{gK08a} is well-adapted to
this situation.  Measures with infinite mass also require us to study
the ideal of square integrable polynomials.  This presents a
difference between one variable and two: there is no reason to study
one variable orthogonal polynomials with respect to a measure with
infinite mass because the ideal of integrable polynomials is a
principal ideal, since all ideals in one variable are.  Our method of
proof also allows us to improve a characterization of two variable
Fej\'er-Riesz factorizations from Geronimo-Woerdeman \cite{GW04}.  We
discuss this below.
\end{remark}

\begin{remark} \label{finiteremark}
The assumption of ``finitely many zeros on $\mathbb{T}^2$'' is there
to put us into the most interesting case and not to avoid a
difficulty.  Every polynomial $q$ with no zeros on the bidisk can be
factored into $q=q_1 q_2$ where $q_1$ has at most finitely many zeros
on the two-torus and every factor of $q_2$ has infinitely many zeros
on the two-torus.  If $q$ has a non-trivial factor of the type $q_2$,
then it can be factored out of the entire sums of squares formula.
These polynomials with no zeros on the bidisk and infinitely many
zeros on the two-torus can be studied separately, and this was done in
Knese \cite{gK08b}.  These notions will appear several places later on
so we give the following definitions of \emph{toral} and
\emph{atoral}.

\begin{definition} \label{def:toral} 
A polynomial $p\in \mathbb{C}[z,w]$ is \emph{toral} if every factor of
  $p$ has infinitely many zeros on $\mathbb{T}^2$.
\end{definition}

\begin{definition} \label{def:atoral} 
A polynomial $p\in \mathbb{C}[z,w]$ is \emph{atoral} if $p$ has
  finitely many zeros on $\mathbb{T}^2$.
\end{definition}

These terms were introduced in Agler-M\mcc Carthy-Stankus \cite{AMS06}
in a more natural way that makes sense for higher dimensions, but
these definitions will suffice for our purposes.

The fact that there are these two types of polynomials makes it
difficult to come up with a limiting argument to prove the sums of
squares formula. We are not suggesting such an argument does not
exist, but any argument that does exist would have to take into
account the difference between polynomials with finitely many zeros on
the bidisk and those with infinitely many.  In any case, it is
preferable to give a unified approach, and this has the added benefit
of introducing notions of ``ideals of orthogonal polynomials'' with
respect to a positive Borel measure.
\end{remark}

\begin{remark}
The requirements on $\mathbf{E}$ and $\mathbf{F}$ in Theorem
\ref{mainthm} that make the decomposition unique are essential in
proving our bounded analytic extension theorem for distinguished
varieties.  The requirements are also curiously asymmetric.  Of course
the roles of $z$ and $w$ can be switched (and the invertibility
requirements switched around).  In fact, the entire formula
\eqref{mainsos} can be ``reflected:'' replace $(z,w)$ with
$(1/\bar{z}, 1/\bar{w})$ and multiply through by $-|z^nw^m|^2$.  The
result will be a new sums of squares formula with $\mathbf{E}$ and
$\mathbf{F}$ replaced with
\[
\refl{\mathbf{E}}(z,w) = z^{n-1}w^m \overline{\mathbf{E}(1/\bar{z},
  1/\bar{w})} \text{ and } \refl{\mathbf{F}}(z,w) = z^{n}w^{m-1}
  \overline{\mathbf{F}(1/\bar{z}, 1/\bar{w})}
\]
respectively. These new choices will have the invertibility
requirements reversed in Theorem \ref{mainthm}.  Notice that in
Example \ref{hardexample} the two choices for the sums of squares
decompositions are \emph{not} simply obtained from one another by
performing this reflection. 

These thoughts beg the following question.  Which polynomials with no
zeros on the bidisk have a unique sums of squares decomposition?
\end{remark}

\begin{theorem} \label{uniquedecompthm} 
Suppose $q \in \mathbb{C}[z,w]$ has no zeros on the bidisk and
  finitely many zeros on the torus. Suppose $q$ has degree $(n,m)$.
  The following are equivalent.
\begin{enumerate}
\item There exist \textbf{unique} non-negative functions $\Gamma_1$,
  $\Gamma_2$ which can be written as the sum of the squared moduli of
  two variable polynomials such that
\begin{equation} \label{generalsos}
|q(z,w)|^2 - |\refl{q}(z,w)|^2 = (1-|z|^2)\Gamma_1(z,w) +
 (1-|w|^2)\Gamma_2(z,w). 
\end{equation}

\item There are \textbf{no} nonzero polynomials $f \in \mathbb{C}[z,w]$
  with degree at most $(n-1,m-1)$ such that
\[
\frac{f}{q} \in L^2(\mathbb{T}^2).
\]

\item There exist vector polynomials $\mathbf{E} \in \mathbb{C}^{n}
[z,w]$ of degree $(n-1,m)$ and $\mathbf{F} \in \mathbb{C}^{m}[z,w]$ of
degree $(n,m-1)$ satisfying
\[
|q(z,w)|^2 - |\refl{q}(z,w)|^2 = (1-|z|^2)|\mathbf{E}(z,w)|^2 +
 (1-|w|^2)|\mathbf{F}(z,w)|^2 
\]
that are \textbf{symmetric} in the sense that:
\[
\mathbf{E}(z,w) = z^{n-1}w^m\overline{\mathbf{E}(1/\bar{z},
1/\bar{w})} \text{ and } \mathbf{F}(z,w) = z^n w^{m-1}
\overline{\mathbf{F}(1/\bar{z}, 1/\bar{w})}
\]
and 
\[
\det E(w) \text{ and } \det F(z)
\]
have all of their zeros on the circle $\mathbb{T}$ where $E \in
\mathbb{C}^{n\times n}[w]$ and $F\in \mathbb{C}^{m\times m}[z]$ are
matrix polynomials described by
\[
\mathbf{E}(z,w) = E(w) \mathbf{\Lambda}_n(z) \text{ and }
\mathbf{F}(z,w) = F(z) \mathbf{\Lambda}_m(w).
\]
\end{enumerate}
\end{theorem}

The polynomial $q(z,w) = 2-z-w$ from Example \ref{easyexample1} has a
unique sums of squares decomposition, since the decomposition we gave
satisfies item (3) above (after multiplying by a suitable unimodular
constant).  Item (2) above says that the polynomials with a unique
decomposition must in some sense have as many zeros as possible on the
torus.  Because of this, polynomials with no zeros on the closed
bidisk never have unique decompositions unless they are one variable
polynomials.

\begin{corollary} \label{uniquedecompcor} 
If $q \in \mathbb{C}[z,w]$ has no zeros on the closed bidisk
  $\cbidisk$, then $q$ has a unique sums of squares decomposition if
  and only if $q$ is a function of only one variable (i.e. one of
  $q$'s partial derivatives vanishes identically).
\end{corollary}

It would be interesting to have a parametrization of the polynomials
in Theorem \ref{uniquedecompthm}. Both Theorem \ref{uniquedecompthm}
and Corollary \ref{uniquedecompcor} are proved in Section
\ref{sec:uniquedecomp}.  

Next, we discuss our applications: two variable Fej\'er-Riesz
factorizations in Section \ref{sec:Fejer}, distinguished varieties in
Section \ref{distvar}, and Agler's Pick interpolation theorem on the
bidisk in Section \ref{sec:agler}. 

The classical Fej\'er-Riesz theorem says that a non-negative one
variable trigonometric polynomial $t$ can be factored as $|p(z)|^2$
where $p \in \mathbb{C}[z]$ has no zeros in the disk $\mathbb{D}$.  It
is false that all non-negative \emph{two} variable trig polynomials
can be factored as $|p(z,w)|^2$ where $p\in \mathbb{C}[z,w]$ has no
zeros on the bidisk.  Indeed, Geronimo and Woerdeman give a
characterization of which \emph{strictly positive} trig polynomials
have a ``Fej\'er-Riesz type factorization'' in \cite{GW04}.  In
Section \ref{sec:Fejer} we give a self-contained proof of Geronimo and
Woerdeman's characterization.  We generalize the characterization to
include non-negative trig polynomials of a certain form.  We also
discuss the importance of the notions of toral and atoral in studying
Fej\'er-Riesz factorizations.

In Section \ref{distvar} we discuss a ``bounded analytic extension
  theorem for polynomials on distinguished varieties.'' We introduce
  this topic now via an example. A \emph{distinguished variety} is a
  special curve in $\mathbb{C}^2$ that exits the bidisk through the
  distinguished boundary.

\begin{example} Consider the following reducible variety in
  $\mathbb{C}^2$
\[
V = \{(z,w) \in \mathbb{C}^2: (z-w)(z^2 - w) = 0\}.
\]
Of interest is the portion of $V$ in the bidisk $V\cap \mathbb{D}^2$
and the analytic functions on $V\cap\mathbb{D}^2$.  The curve $V$ is
an example of a distinguished variety. Like all distinguished
varieties it has a ``determinantal representation'' of the following
form:
\[
V\cap \mathbb{D}^2 = \{(z,w) \in \mathbb{D}^2: \det(wI-\Phi(z) )=0\}
\]
where $\Phi$ is a rational matrix valued inner function.  In this case
$\Phi$ can be taken to be the $2\times 2$ matrix function
\[
\Phi(z) = \frac{1}{2} \begin{pmatrix} z(1+z) & z^2(1-z) \\ (1-z) &
  z(1+z) \end{pmatrix}
\]
and saying $\Phi$ is inner just means
\[
\Phi(z)\Phi(z)^* = I_2 \quad \text{ for } z \in \mathbb{T}.
\]
As can easily be checked
\[
\det(wI_2 - \Phi(z)) = w^2 - zw - z^2w + z^3 = (w-z)(w-z^2).
\]
This example, while simple, is instructive because it has a
singularity at the origin and more importantly a singularity on the
torus at the point $(1,1)$.

From the work in Knese \cite{gK08b}, we can associate to $V$ a
polynomial with no zeros on the bidisk and a single zero on
$\mathbb{T}^2$ at $(1,1)$ and use our sums of squares decomposition to
provide extra details about this determinantal representation.
Namely, $V$ is defined as the zero set of $p(z,w) = w^2-zw-z^2w+z^3$,
and if we define
\[
q(z,w) = z^3p(1/z,w) = w^2z^3-z^2w-zw+1
\]
then 
\[
\frac{\partial q}{\partial w} = 2wz^3-z^2-z \text{ and}
\]
\[
\refl{\frac{\partial q}{\partial w}} = 2-zw-z^2w
\]
is the associated polynomial with no zeros on $\mathbb{D}^2$ and a
single zero on $\mathbb{T}^2$ at $(1,1)$.  This is the polynomial from
Example \ref{hardexample}, and as we have seen, there are many ways
to write a sums of squares formula for it.  Our main theorem, Theorem
\ref{mainthm}, guarantees that it has a decomposition with certain
extra invertibility constraints.  

Leaving out the details, we can prove $\Phi$ has a ``polynomial
eigenvector'' $\mathbf{Q}(z,w)$. By this we mean
\[
\Phi(z)\mathbf{Q}(z,w) = w\mathbf{Q}(z,w)
\]
for all $(z,w) \in V$.  In this example, we can take
\[
\mathbf{Q}(z,w) = \begin{pmatrix} 2w - z -z^2 \\ 1-z \end{pmatrix}
\]
and this $\mathbf{Q}$ has the special property that when we write
\[
\mathbf{Q}(z,w) = Q(z) \begin{pmatrix} 1 \\ w \end{pmatrix}
\]
where
\[
Q(z) = \begin{pmatrix} -z-z^2 & 2 \\ 1-z & 0 \end{pmatrix}
\]
we have that $Q(z)$ is invertible in
$\overline{\mathbb{D}}\setminus\{1\}$.  This is significant because we
can prove an analytic extension theorem using this as follows.

Let $f\in \mathbb{C}[z,w]$ which we think of as a function on $V$.
Then, the rational function
\[
F(z,w) = (1,0) Q(z)^{-1}f(zI,\Phi(z))\mathbf{Q}(z,w)
\]
agrees with $f$ on $V$ because $\mathbf{Q}$ is a polynomial
eigenvector for $\Phi$ on $V$.  Furthermore, the size of $F$ on the
bidisk can be estimated purely in terms of a fixed rational function
of $z$ and the supremum of $f$ on $V\cap\mathbb{D}^2$.

Indeed, 
\[
\begin{aligned}
|F(z,w)| &\leq |(1,0)Q(z)^{-1}||\mathbf{Q}(z,w)| \sup_{V\cap\mathbb{D}^2} |f| \\
& \leq \sqrt{1+ \frac{16}{|1-z|^2}} \sup_{V\cap\mathbb{D}^2} |f|.
\end{aligned}
\]
\EOEx
\end{example}

More generally, we have the following theorem; a more detailed version
is presented in Section \ref{distvar} as Theorem \ref{extendthm}.

\begin{theorem} \label{extendthmsmall} Let $V\subset \mathbb{C}^2$ be
  a distinguished variety.  Then, there is a rational function of $z$,
  $C(z)$, with no poles in $\mathbb{D}$, such that for every $f \in
  \mathbb{C}[z,w]$, there is a rational function $F\in
  \mathbb{C}(z,w)$, holomorphic on $\mathbb{D}^2$, which agrees with
  $f$ on $V$:
\[
F(z,w) = f(z,w) \text{ for all } (z,w) \in V\cap \mathbb{D}^2
\]
and satisfies the estimate
\[
|F(z,w)| \leq |C(z)| \sup_{V\cap\mathbb{D}^2}|f|
\]
for all $(z,w) \in \mathbb{D}^2$.

If $V$ has no singularities on $\mathbb{T}^2$, $C(z)$ can be taken to
be a constant.
\end{theorem}

Finally, in Section \ref{sec:agler} we give a new proof of necessity
in Agler's Pick interpolation via the sums of squares formula.  

\section{Outline of what follows}
In the next section, we present some background lemmas on sums of
squares decompositions.  We also prove the uniqueness portion of the
main theorem.  After that we dive into the details of the paper.  In
Section \ref{prelim}, we present most of the notation and machinery
for the paper.  Much of the paper involves getting into the
intricacies of subspaces of two variable polynomials.  We have found
that a pictorial notation introduced in Knese \cite{gK08a} is useful
for thinking about these subspaces, although we admit it takes some
getting used to.

In Section \ref{genprop} we study two variable orthogonal polynomials
on the torus $\mathbb{T}^2$ with respect to a positive Borel measure
$\mu$, which may have infinite mass.  We present a two variable
Christoffel-Darboux formula with a certain ``error term,'' which
prevents it from being a straightforward generalization of the
Christoffel-Darboux formula for orthogonal polynomials on the unit
circle.  This error term disappears when $\mu$ satisfies an added
orthogonality condition.  In Section \ref{sec:ao}, we explore the
implications of this ``orthogonality condition'' and prove a spectral
matching result for orthogonal polynomials on $\mathbb{T}^2$. (In the
context of \emph{probability} measures, all of this was done by
Geronimo and Woerdeman in \cite{GW04}.) In Section \ref{sec:BernSzeg},
we go in the reverse and prove a special class of positive Borel
measures satisfy this orthogonality condition.  The class consists of
Lebesgue measure on $\mathbb{T}^2$ weighted with $1/|p|^2$ where $p$
is a polynomial with no zeros on $\mathbb{D}^2$ and finitely many
zeros on $\mathbb{T}^2$; the so-called Bernstein-Szeg\H{o} measures.

So, the main theorem follows by taking a polynomial with no zeros on
the bidisk, defining a Bernstein-Szeg\H{o} measure, observing that it
satisfies the aforementioned ``orthogonality condition,'' and then
writing down a detailed Christoffel-Darboux formula, which is our
desired sums of squares formula.  This is spelled out in Section
\ref{sec:mainthm}.

In Sections \ref{sec:uniquedecomp}, \ref{sec:Fejer}, \ref{distvar},
and \ref{sec:agler} we present applications (discussed
above). Finally, we conclude with some general questions in Section
\ref{questions}. There is a notational index at the end of the paper. 

\section{Sums of squares and uniqueness}
In this section we present several lemmas on sums of squares
decompositions.  Lemma \ref{uniquenesslemma} proves uniqueness in
Theorem \ref{mainthm}, namely item (4).  

The following theorem can be found in D'Angelo \cite{jD93}.
\begin{theorem}[Polarization for holomorphic functions] 
Let $\Omega$ be a domain in $\mathbb{C}^N$ and set $\Omega^* = \{\bar{z} =
(\bar{z}_1,\dots, \bar{z}_N): z \in \Omega\}$. If $f:\Omega \times
\Omega^* \to \mathbb{C}$ is a holomorphic function with the property
that
\[
f(z,\bar{z}) = 0 \text{ for all } z \in \Omega
\]
then
\[
f(z,w) = 0 \text{ for all } (z,w) \in \Omega\times \Omega^*.
\]
\end{theorem}

The following lemma holds equally well for multi-variable polynomials,
and may be well known to some readers.

\begin{lemma} \label{simpleunique} Suppose $\Gamma(z)$ is a sum of
  squares of polynomials.  Then, there exists a positive integer $N$
  and polynomials $A_1, \dots, A_N \in \mathbb{C}[z]$ such that
\[
\Gamma(z) = \sum_{j=1}^N |A_j(z)|^2
\]
and if $\Gamma$ can be written as
\[
\Gamma(z) = \sum_{k=1}^M |B_k(z)|^2
\]
for some polynomials $B_1, \dots, B_M \in \mathbb{C}[z]$, then $N\leq
M$ and there exists an isometric $M\times N$ matrix $V$ (i.e. $V^{*}V
=I_N$) such that
\[
V \begin{pmatrix} A_1(z) \\ \vdots \\ A_N(z) \end{pmatrix} =
\begin{pmatrix} B_1(z) \\ \vdots \\ B_M(z) \end{pmatrix}.
\]
\end{lemma}

\begin{proof} By assumption there exist polynomials $C_1, \dots, C_L
  \in \mathbb{C}[z]$ such that
\[
\Gamma(z) = \sum_{j=1}^L |C_j(z)|^2.
\]
Define $\mathbf{C}(z) = (C_1(z), \dots, C_L(z))^t \in \mathbb{C}^L[z]$
and let
\[
\mathcal{K} = \text{span} \{\mathbf{C}(z): z\in \mathbb{C}\}
\]
and 
\[
N = \dim \mathcal{K}. 
\]
Let $U: \mathcal{K} \to \mathbb{C}^N$ be any isometry. Of course, $U$
can be extended to all of $\mathbb{C}^L$ by mapping elements of
$\mathbb{C}^L \ominus \mathcal{K}$ to $\mathbf{0}$ and can therefore
be viewed as an $N\times L$ matrix.  Define
\[
\mathbf{A}(z) = U \mathbf{C}(z)
\]
and write $\mathbf{A}(z) = (A_1(z), \dots, A_N(z))^t$.  Then,
\[
\Gamma(z) = |\mathbf{C}(z)|^2 = |\mathbf{A}(z)|^2 = \sum_{j=1}^N
|A_j(z)|^2.
\]
Now, suppose there are polynomials $B_1, \dots, B_M \in \mathbb{C}[z]$
such that
\[
\sum_{k=1}^M |B_k(z)|^2 = \sum_{j=1}^N |A_j(z)|^2.
\]
By the polarization theorem for holomorphic functions, 
\[
\sum_{k=1}^M B_k(z)\overline{B_k(Z)} = \sum_{j=1}^N A_j(z)
\overline{A_j(Z)}
\]
for all $z, Z \in \mathbb{C}$.  This can be rewritten more compactly
as
\begin{equation} \label{polarized}
\ip{\mathbf{B}(z)}{\mathbf{B}(Z)} = \ip{\mathbf{A}(z)}{\mathbf{A}(Z)}
\end{equation}
(the inner product on the left is on $\mathbb{C}^M$ and the inner
product on the right is on $\mathbb{C}^N$).  

The map $V: \mathbb{C}^N \to \mathbb{C}^M$ which
sends
\[
\sum_{j=1}^R c_j \mathbf{A}(z_j) \mapsto \sum_{j=1}^R c_j \mathbf{B}(z_j)
\]
for any points $z_1, z_2,\dots, z_R$ and any scalars
$c_1, \dots, c_R$, is well-defined, linear, and isometric since
\[
\begin{aligned} |\sum_{j=1}^R c_j \mathbf{A}(z_j)|^2 &= \sum_{j,k} c_j\bar{c}_k
\ip{\mathbf{A}(z_j)}{\mathbf{A}(z_k)} \\
&= \sum_{j,k} c_j\bar{c}_k
\ip{\mathbf{B}(z_j)}{\mathbf{B}(z_k)} = |\sum_{j=1}^R c_j \mathbf{B}(z_j)|^2
\end{aligned}
\]
by \eqref{polarized} and since $\text{span} \{ \mathbf{A}(z): z\in
\mathbb{C}\} = \mathbb{C}^N$ by construction of $\mathbf{A}$.  So, $V$
may be thought of as an $M\times N$ isometric matrix satisfying
\[
\mathbf{B}(z) = V \mathbf{A}(z).
\]
for all $z \in \mathbb{C}$.  This implies $N\leq M$.  
\end{proof}

\begin{lemma} \label{hardunique}
Suppose $\mathbf{E} \in \mathbb{C}^n[z,w]$ has degree at most
$(n-1,m)$ and has the property that when we write 
\[
\mathbf{E}(z,w) = E(w) \bLam_n(z) \qquad E(w) \in \mathbb{C}^{n\times n}[w],
\]
the matrix polynomial $E(w)$ is invertible for all $w \in \mathbb{D}$.
Suppose further that $\mathbf{A} \in \mathbb{C}^N[z,w]$ is a two
variable matrix polynomial satisfying
\[
|\mathbf{E}(z,w)|^2 =|\mathbf{A}(z,w)|^2 \text{ for } (z,w) \in
 \mathbb{C}\times \mathbb{T}. 
\]
Then, $n\leq N$, $\mathbf{A}(z,w)$ has degree at most $n-1$ in $z$ and
there exists an $N\times n$ matrix valued rational inner function
$\Psi: \mathbb{D} \to \mathbb{C}^{N\times n}$, holomorphic on
$\mathbb{D}$ such that
\[
\mathbf{A}(z,w) = \Psi(w) \mathbf{E}(z,w).
\]
\end{lemma}

By ``$N\times n$ matrix valued inner function'' we mean that $\Psi$ is
isometry valued on the circle (or more appropriately, unitary valued
in the case $n=N$).

\begin{proof}
We have assumed 
\[
|\mathbf{E}(z,w)|^2 = |\mathbf{A}(z,w)|^2
\]
for all $z \in \mathbb{C}$ but $w\in \mathbb{T}$.  By the polarization
theorem for holomorphic functions
\begin{equation} \label{EApolar}
\ip{\mathbf{E}(z,w)}{\mathbf{E}(Z,w)} =
\ip{\mathbf{A}(z,w)}{\mathbf{A}(Z,w)} 
\end{equation}
for all $z,Z \in \mathbb{C}$ and $w \in \mathbb{T}$.  The left hand
side has degree at most $n-1$ in $z$ and this implies
$\mathbf{A}(z,w)$ has degree at most $n-1$ in $z$.  (If some component
with the largest degree, say $A_1(z,w)=\sum_{j=0}^{M}a_j(w) z^j$, of
$\mathbf{A}(z,w)$ has degree $M$ larger than $n-1$, then
\[
A_1(z,w)\overline{A_1(Z,w)} = |a_M(w)|^2 z^M \bar{Z}^M + \text{ lower
  order terms }
\]
and we would necessarily have $a_M(w) \equiv 0$ on $\mathbb{T}$.  This
would imply $a_M(w)\equiv 0$ for all $w \in \mathbb{C}$.) Therefore,
we may write
\[
\mathbf{A}(z,w) = A(w) \bLam_n(z)
\]
where $A(w)$ is an $N\times n$ matrix polynomial.  Rewriting
\eqref{EApolar} in matrix form we have
\[
\bLam_n(Z)^* E(w)^* E(w) \bLam_n(z) = \bLam_n(Z)^* A(w)^* A(w) \bLam_n(z)
\]
and since this holds for all $z, Z \in \mathbb{C}$ 
\begin{equation} \label{EAoncircle}
E(w)^* E(w) = A(w)^* A(w)
\end{equation}
for all $w \in \mathbb{T}$ because $\bLam_n(z)$ spans $\mathbb{C}^n$
as $z$ varies over any $n$ points.  Now define 
\[
\Psi(w) = A(w) E(w)^{-1}
\]
for $w \in \mathbb{D}$. This is a rational matrix polynomial with no
poles on the disk since $E(w)$ is invertible in the disk.  Equation
\eqref{EAoncircle} says that $\Psi(w)$ is isometric for $w \in
\mathbb{T}$.  In particular, $n\leq N$, any singularities of $\Psi$ on
the circle are removable ($\Psi$ is rational and bounded on the
circle), and by the maximum principle $\Psi$ is contraction valued in
the disk.  By definition,
\[
\mathbf{A}(z,w) = \Psi(w) \mathbf{E}(z,w)
\]
for all $z,w \in \mathbb{C}$.
\end{proof}

\begin{lemma}[Uniqueness Lemma] \label{uniquenesslemma} Suppose $\mathbf{E},
  \mathbf{\tilde{E}} \in \mathbb{C}^n[z,w]$ have degree at most
  $(n-1,m)$ and have the property that when written in terms of
  $n\times n$ matrix polynomials $E, \tilde{E} \in \mathbb{C}^{n\times
    n} [w]$ as
\[
\mathbf{E}(z,w) = E(w) \mathbf{\Lambda}_n(z)
\qquad 
\mathbf{\tilde{E}}(z,w) = \tilde{E}(w) \mathbf{\Lambda}_n(z)
\]
both $E(w)$ and $\tilde{E}(w)$ are invertible for all $w \in
\mathbb{D}$.  Suppose further that there are vector polynomials
$\mathbf{F}, \mathbf{\tilde{F}} \in \mathbb{C}^m[z,w]$ such that
\begin{align}
& (1-|z|^2) |\mathbf{E}(z,w)|^2 + (1-|w|^2) |\mathbf{F}(z,w)|^2
  \nonumber \\
&=  (1-|z|^2) |\mathbf{\tilde{E}}(z,w)|^2 +
(1-|w|^2)|\mathbf{\tilde{F}}(z,w)|^2 \label{EEtilde}
\end{align}

Then, there
exists an $n\times n$ unitary $U_1$ and an $m\times m$ unitary $U_2$
such that 
\[
\mathbf{E}(z,w) = U_1 \mathbf{\tilde{E}}(z,w) \qquad \mathbf{F}(z,w) =
U_2 \mathbf{\tilde{F}}(z,w).
\]
\end{lemma}

\begin{proof} 
  Setting $|w|=1$ in \eqref{EEtilde} and canceling the factor
  $(1-|z|^2)$ we have
\[
|\mathbf{E}(z,w)|^2 = |\mathbf{\tilde{E}}(z,w)|^2 \text{ for } (z,w)
 \in \mathbb{C}\times \mathbb{T}.
\]
Both $\mathbf{E}$ and $\mathbf{\tilde{E}}$ satisfy the conditions of
Lemma \ref{hardunique}.  Therefore, there exist $n\times n$ matrix
valued rational inner functions $\Psi_1, \Psi_2: \mathbb{D} \to
\mathbb{C}^{n\times n}$ such that
\[
\begin{aligned}
\mathbf{\tilde{E}}(z,w) &= \Psi_1(w) \mathbf{E}(z,w) \text{ and}\\
\mathbf{E}(z,w) &= \Psi_2(w) \mathbf{\tilde{E}}(z,w).
\end{aligned}
\]
This implies $\Psi_1(w) \Psi_2(w) = I$ and since $\Psi_1, \Psi_2$ are
contractive valued, this can only occur when both $\Psi_1$ and
$\Psi_2$ are constant and equal to unitary matrices.  Hence, there
exists an $n\times n$ unitary matrix $U_1$ such that 
\[
\mathbf{E}(z,w) = U_1 \mathbf{\tilde{E}}(z,w).
\]
This implies 
\[
|\mathbf{E}(z,w)|^2 = |\mathbf{\tilde{E}}(z,w)|^2
\]
for all $(z,w) \in \mathbb{C}^2$.  

In turn, by \eqref{EEtilde} we have
\[
|\mathbf{F}(z,w)|^2 = |\mathbf{\tilde{F}}(z,w)|^2
\]
for all $(z,w) \in \mathbb{C}^2$.  By Lemma \ref{simpleunique}, there
exists an $m\times m$ unitary matrix $U_2$ such that
\[
\mathbf{F}(z,w) = U_2 \mathbf{\tilde{F}}(z,w).
\]
\end{proof}

\section{Preliminaries} \label{prelim}
As in Knese \cite{gK08a}, our approach will be to study two variable
orthogonal polynomials with respect to a positive Borel measure $\mu$
on the two-torus.  The difference is that here we allow measures with
infinite mass. In particular, we study ``Bernstein-Szeg\H{o}''
measures on $\mathbb{T}^2$
\[
\frac{1}{|q(z, w)|^2} d\sigma 
\]
where $d\sigma$ is normalized Lebesgue measure on the torus:
\begin{equation} \label{Lebesgue}
d\sigma = d\sigma(z,w) = \frac{dz}{2\pi iz} \frac{dw}{2\pi iw}
\end{equation}
and $q \in \mathbb{C}[z,w]$ has finitely many zeros on $\mathbb{T}^2$
(and hence this measure can have infinite mass).  On one hand, this
causes a number of certain superficial (but still interesting) changes
in the theory. For instance, we have to deal with the ideal
$\mathbb{C}[z,w]\cap L^2(\mu)$ of polynomials in $L^2(\mu)$ as opposed
to all of $\mathbb{C}[z,w]$ when studying orthogonal polynomials.  (In
particular, studying moment matrices will not be an option, because
our measures may not have finite moments.) On the other hand, this
change forces us to take greater care in certain situations.  For
instance, if $q\in \mathbb{C}[z,w]$ has no zeros on the bidisk and
finitely many zeros on the two-torus, we \emph{cannot} say (as we
would in the case with no zeros on $\mathbb{T}^2$) that
\[
\int_{\mathbb{T}^2} \frac{1}{q(z, w)} d\sigma(z,w) = \frac{1}{q(0,0)}
\]
since $1/q$ will not be integrable.  Perhaps this integral could be
understood in a principal value sense, however we confront this issue
in our own way in Proposition \ref{qrelations}.

Let us begin to provide some details.  We shall make the following
standing assumptions
\begin{itemize}
\item $\mu$ is a positive Borel measure on $\mathbb{T}^2$,

\item the ideal 
\begin{equation} \label{def:ideal}
\mathcal{I}_\mu := L^2(\mu) \cap \mathbb{C}[z,w]
\end{equation}
 is nonempty, where elements of $\mathbb{C}[z,w]$ here are thought of
as measurable functions on $\mathbb{T}^2$,

\item the support of $\mu$ is not contained in the intersection of the
  zero set of a nonzero polynomial with the two-torus $\mathbb{T}^2$.
  This ensures that $||q||_{L^2(\mu)} \ne 0$ if $q \ne 0$.
\end{itemize}

\begin{definition} \label{def:degree}
  If $j,k$ are nonnegative integers, we say $q \in \mathbb{C}[z,w]$
  has \emph{degree $(j,k)$} and we write
\[
\deg (q) = (j,k)
\]
if $q$ has degree $j$ in $z$ and $k$ in $w$.  Also, $q$ has
\emph{degree at most $(j,k)$} if $q$ has degree at most $j$ in $z$ and
at most $k$ in $w$, in which case we write
\[
\deg (q) \leq (j,k).
\]
\end{definition}

Given $q \in \mathbb{C}[z,w]$ we use 
\begin{equation} \label{def:hat}
\hat{q}(j,k)
\end{equation}
 to denote the coefficient of $z^jw^k$ in the Fourier series of $q$.

\begin{remark} \label{fixnm} 
Throughout the article, we fix positive integers $n$ and $m$. The
notations below depend on this.
\end{remark}

We use the following notations as in Knese \cite{gK08a} which define
subspaces of polynomials based on what frequencies may appear in their
Fourier series (or in other language, we define subspaces based on the
\emph{carrier} of the polynomials).  The symbols should be thought of
a lying in the grid $\mathbb{Z}^2$ with the lower left corners
representing the origin.

\begin{notation} \label{boxnotation}
\[
\begin{aligned}
  \gkbox &:= \{q \in \mathbb{C}[z,w] : \deg (q) \leq (n,m)\} \\
  \gkboxr &:= \{ q\in \mathbb{C}[z,w] : \deg (q) \leq (n-1,m)\}\\
  \gkboxu &:= \{ q \in \mathbb{C}[z,w] : \deg (q) \leq (n,m-1) \} \\
  \gkboxsm &:= \{ q \in \mathbb{C}[z,w] : \deg (q) \leq (n-1, m-1) \} \\
  \gkboxll &:= \{ q \in \gkbox : q(0,0) = 0\}\\
  \gkboxur &:= \{ q \in \gkbox : \hat{q}(n,m) = 0 \}
\end{aligned}
\]

For any of the above subspaces (and similar variations) we shall use a
subscript $\mu$ to denote the intersection with $L^2(\mu)$.  Namely,
\[
\begin{aligned}
\gkbox_\mu &:= \gkbox \cap L^2(\mu) \\
\gkboxr_\mu &:= \gkboxr \cap L^2(\mu) \\
\gkboxu_\mu &:= \gkboxu \cap L^2(\mu), \text{ et cetera}\dots
\end{aligned}
\]
\end{notation}

We continue Example \ref{easyexample1} to make all of the above
definitions concrete.
\begin{example} \label{easyexample2} Let $q(z,w) = 2-z-w$. Let
\[
d\mu = \frac{1}{|2-z-w|^2} d\sigma(z,w) = \frac{1}{(2\pi i)^2|2-z-w|^2}
\frac{dz}{z} \frac{dw}{w}.
\]

It turns out that $\mathcal{I}_\mu = L^2(\mu) \cap \mathbb{C}[z,w]$
equals the maximal ideal $(z-1, w-1) \subset \mathbb{C}[z,w]$.  We do
not think this is obvious since
\[
\frac{z-1}{2-z-w}
\]
is unbounded in the bidisk (set $(z,w) = (t+i\sqrt{1-t},
t-i\sqrt{1-t})$ and see what happens when $t\nearrow 1$). Let us
provide some details.

\begin{claim} 
\[
1 \notin L^2(\mu) \text{ and } z-1, w-1 \in L^2(\mu)
\]
\end{claim}
\begin{proof}
  It is easiest to compute the radial integral means:
\begin{equation} \label{integrals}
\frac{1}{(2\pi i)^2} \int_{r\mathbb{T}^2} \frac{|z-1|^2}{|2- z- w|^2}
  \frac{dw}{w}\frac{dz}{z} = \frac{1}{(2\pi)^2} \int_{0}^{2\pi} \int_{0}^{2\pi}
  \frac{|r e^{i\theta} -1|^2}{|2- re^{i\theta} - re^{i\phi}|^2} \,
  d\phi d\theta
\end{equation}
where $r\mathbb{T}^2 = (r\mathbb{T})\times(r\mathbb{T})$.  
Recall
\[
\int_{0}^{2\pi} \frac{1}{|1-Z e^{i\phi}|^2} \frac{d\phi}{2\pi} =
\frac{1}{1-|Z|^2} 
\]
for any $Z \in \mathbb{D}$.
So, the inner integral of \eqref{integrals} equals
\[
\begin{aligned}
  \frac{|re^{i\theta}-1|^2}{2\pi |2-re^{i\theta}|^2} &\int_{0}^{2\pi}
  \frac{1}{|1- \frac{r}{2-re^{i\theta}} e^{i\phi}|^2} d\phi =
  \frac{|re^{i\theta}-1|^2}{|2-r
    e^{i\theta}|^2} \frac{1}{1 - \frac{r^2}{|2-re^{i\theta}|^2}} \\
  & = \frac{|re^{i\theta}-1|^2}{|2-re^{i\theta}|^2 - r^2}\\
  &= \frac{|re^{i\theta}-1|^2}{|(\sqrt{1+r}+\sqrt{1-r}) -
    (\sqrt{1+r}-\sqrt{1-r}) e^{i\theta}|^2}
\end{aligned}
\]
where this last expression comes from a Fej\'{e}r-Riesz type of
factorization of the denominator (but can also just be verified
directly). Our integral reduces to
\[
\begin{aligned}
& \frac{1}{2\pi} \int_{0}^{2\pi}
    \frac{|1-re^{i\theta}|^2}{|(\sqrt{1+r}+\sqrt{1-r}) -
    (\sqrt{1+r}-\sqrt{1-r}) e^{i\theta}|^2} d\theta\\ & = \frac{1+r^2
    - 2r(\frac{\sqrt{1+r} - \sqrt{1-r}}{\sqrt{1+r} +
    \sqrt{1-r}})}{4\sqrt{1-r^2}} \\
& = \frac{2-\sqrt{1-r^2}}{4} \to \frac{1}{2} \qquad \text{ as } r
\nearrow 1.
\end{aligned}
\]
Hence, $\frac{z-1}{2-z-w} \in H^2(\mathbb{T}^2)$ and therefore $z-1
\in L^2(\mu)$.  
Similarly, $w-1 \in L^2(\mu)$.

To prove $1 \notin L^2(\mu)$ we can use some of the above computations to
prove
\[
\frac{1}{(2\pi i)^2} \int_{r\mathbb{T}^2} \frac{1}{|2-z-w|^2} \gkmeas
= \frac{1}{4\sqrt{1-r^2}} \to \infty \text{ as } r\nearrow 1.
\]
\end{proof}

If we set $n=1$ and $m=1$, then
\[
\begin{aligned}
\gkboxsm_\mu &= \{0\} \\
\gkboxr_\mu &= (w-1)\mathbb{C}\\
\gkboxu_\mu &= (z-1)\mathbb{C}\\
\gkbox_\mu &= \text{span}\{z-1, w-1, z+w-2zw\}
\end{aligned}
\]

 We now return to the general situation.
\EOEx
\end{example}

The inner product on $L^2(\mu)$ will be denoted by
\begin{equation} \label{innerprod}
\ip{f}{g}_\mu = \int_{\mathbb{T}^2} f\bar{g} d\mu.
\end{equation}

We shall make use of the machinery of reproducing kernel Hilbert
spaces. 

\begin{notation} \label{kernelnotation} 
Given a finite dimensional subspace $V \subset L^2(\mu)\cap
  \mathbb{C}[z,w]$, we shall use $KV$ to denote the reproducing kernel
  of $V$.  Namely, for each $(Z,W) \in \mathbb{C}^2$, $KV_{(Z,W)}$ is
  the unique element of $V$ satisfying
\[
f(Z,W) = \ip{f}{KV_{(Z,W)}}_\mu
\]
for all $f \in V$ and we define $KV:\mathbb{C}^2 \times \mathbb{C}^2
\to \mathbb{C}$ by
\[
KV((z,w), (Z,W)) := KV_{(Z,W)} (z,w).
\]
\end{notation} 

It is not hard to show $KV$ is conjugate symmetric:
\[
KV((z,w),(Z,W)) = \overline{KV((Z,W), (z,w))},
\]
and if $\{e_1,\dots, e_N\}$ is an orthonormal basis of $V$, then
\[
KV((z,w),(Z,W)) = \sum_{j=1}^{N} e_j(z,w) \overline{e_j(Z,W)}.
\]

We use the following notations for shifts and certain
orthogonal complements using the inner product on $L^2(\mu)$.
\begin{notation} \label{complementnotation}
\[
\begin{aligned}
  w \gkboxsm_\mu &:= \{w p: p \in \gkboxsm_\mu\} \\
  z \gkboxsm_\mu &:= \{z p : p \in \gkboxsm_\mu\} \\
  \gkboxrperpdn_\mu &:= \gkboxr_\mu \ominus \gkboxsm_\mu \\
  \gkboxrperpup_\mu &:= \gkboxr_\mu \ominus (w\gkboxsm_\mu) \\
  \gkboxuperplt_\mu &:= \gkboxu_\mu \ominus \gkboxsm_\mu\\
  \gkboxuperprt_\mu &:= \gkboxu_\mu \ominus (z\gkboxsm_\mu) \\
  \gkboxurperpdn_\mu &:= \gkboxur_\mu \ominus \gkboxu_\mu \\
  \gkboxllperpup_\mu &:= \gkboxll_\mu \ominus (w\gkboxu_\mu) \\ 
  \gkboxllperp_\mu &:= \gkbox_\mu \ominus \gkboxll_\mu \\
  \gkboxurperp_\mu &:= \gkbox_\mu \ominus \gkboxur_\mu 
\end{aligned}
\]
\end{notation}

These last two subspaces are especially important.  They are either
one dimensional or trivial and will provide the connection between
reproducing kernels and polynomials with no zeros on the bidisk.

Frequent use will be made of the following notion of polynomial
``reflection.''

\begin{definition} \label{def:reflection} 
If $p \in \mathbb{C}[z,w]$ is a polynomial of degree at most $(j,k)$
  we define the \emph{reflection} (at the $(j,k)$ degree) to be
\[
\refl{p}(z,w) := z^j w^k \overline{p(1/\bar{z}, 1/\bar{w})}.
\]
\end{definition}

We conclude this section with a lemma about the presence of zeros on
the ``undistinguished'' portion of the boundary of $\mathbb{D}^2$,
namely $(\mathbb{D}\times \mathbb{T})\cup(\mathbb{T} \times
\mathbb{D})$.

\begin{lemma} \label{zeroslemma} Suppose $q \in \mathbb{C}[z,w]$ has
  no zeros on $\mathbb{D}^2$.  If $q(z_0,w_0)=0$ for some $(z_0,w_0)
  \in \mathbb{T}\times \mathbb{D}$, then $q(z_0, w) = 0$ for all $w
  \in \mathbb{C}$; i.e. $(z-z_0)$ divides $q$.  In particular, there
  can only be finitely many $z_0 \in \mathbb{T}$ such that $q(z_0,
  \cdot)$ has a zero in $\mathbb{D}$.
\end{lemma}

\begin{proof} There is no harm in assuming $q$ is irreducible. Suppose
  $q(z_0,w)$ is not identically zero as a function of $w$.  Then, we
  can apply the Weierstrass preparation theorem to $q$ and write
\[
q(z,w) = u(z,w)(z^k + a_1(w) z^{k-1} + \cdots + a_k(w))
\]
on some bidisk $D_1\times D_2$ containing $(z_0,w_0)$ where $u$ is
holomorphic and nonvanishing on $D_1\times D_2$ and each $a_j$ is
holomorphic on $D_2$.  We also assume $D_2 \subset \mathbb{D}$.
Furthermore, for $w \in D_2\setminus\{w_0\}$, each $a_j(w)$ is a
symmetric function of the $k$ (necessarily) distinct roots (by
irreducibility) $z_1(w), z_2,(w), \dots, z_k(w) \in D_1$ of $q(\cdot,
w)$ for $w \in D_2\setminus \{w_0\}$. Note $a_k(w) = (-1)^kz_1(w)
\cdots z_k(w)$ for $w\ne w_0$ and $a_k(w_0) = (-z_0)^k$.  Since $q$
has no zeros in $\mathbb{D}^2$, $|z_j(w)| \geq 1$ for all $j$ and $w
\in D_2$, and hence $|a_k(w)|\geq 1$ for all $w \in D_2$. Since
$|a_k(w_0)|=1$ the maximum principle implies $a_k$ is a unimodular
constant, which in turn implies the roots $z_1(w), \dots, z_k(w)$ are
all unimodular valued.  This can only be the case if they are constant
and equal to $z_0$; i.e. $q(z,w)$ can be divided by $z-z_0$.
\end{proof}

\section{General properties of orthogonal polynomials on
  $\mathbb{T}^2$} \label{genprop}

This section is about orthogonal polynomials on $\mathbb{T}^2$ with
respect to a (not necessarily finite) positive Borel measure on
$\mathbb{T}^2$. We use reproducing kernels to study entire subspaces
of polynomials all at once, so the ``orthogonal polynomials'' are in
some sense disguised.  The following theorem expresses certain
rearrangements of the subspaces described in the previous section
using reproducing kernels.

\begin{theorem} \label{rearrangethm} Let $\mu$ be a positive Borel
  measure on $\mathbb{T}^2$ for which $\mathbb{C}[z,w] \cap L^2(\mu)
  \ne \varnothing$ and for which $\gkboxllperp_\mu$ is one
  dimensional.  Let
\[
\epsilon := (K \gkboxurperpdn_\mu - K \gkboxrperpdn_\mu) -
(K\gkboxllperpup_\mu - K \gkboxlperpup_\mu ).
\]

If $q$ is any unit norm polynomial in
  $\gkboxllperp_\mu$, then writing $q\bar{q} = q(z,w)
  \overline{q(Z,W)}$ and omitting the expressions ``$((z,w),(Z,W))$'' 
\[
\begin{aligned}
  q\bar{q} &- \refl{q} \overline{\refl{q}} \\
  =& (1-z \bar{Z})(1-w \bar{W}) K\gkboxsm_\mu  \\
  & + (1-z \bar{Z})K\gkboxrperpdn_\mu + (1-w\bar{W}) K
  \gkboxuperplt_\mu +\epsilon \\
  =& (1-z \bar{Z}) K\gkboxrperpdn_\mu + (1-w \bar{W})
  K\gkboxuperprt_\mu + \epsilon\\
=& (1-z \bar{Z}) K\gkboxrperpup_\mu + (1-w \bar{W})
  K\gkboxuperplt_\mu + \epsilon\\
%
\end{aligned}
\]

\end{theorem}

The proof of this theorem is identical to the proof of Theorem 4.5 in
Knese \cite{gK08a}, which is for probability measures, so we omit
it. All that is needed for the proof to work is the fact that
reflection and multiplication by a coordinate function are both
isometric operations in $L^2(\mu)$ and that these operations behave
nicely with respect to reproducing kernels (i.e. reflecting a subspace
reflects the reproducing kernels, the reproducing kernel of an
orthogonal direct sum of two subspaces is the sum of the reproducing
kernels of the two subspaces, and multiplying a subspace by $z$
multiplies the reproducing kernel by $z\bar{Z}$).

The above formula may appear complicated but when
\[
\gkboxurperpdn_\mu = \gkboxrperpdn_\mu
\]
we have
\[
\gkboxllperpup_\mu = \gkboxlperpup_\mu
\]
by reflecting these subspaces and this implies that the $\epsilon$
above disappears.  Several nice things occur because of this. We
devote Section \ref{sec:ao} to studying what happens when
$\gkboxurperpdn_\mu = \gkboxrperpdn_\mu$, culminating in the fact in
Section \ref{sec:BernSzeg} that on $\gkbox_\mu$, $\mu$ behaves like a
\emph{Bernstein-Szeg\H{o} measure}:
\[
\frac{1}{|q(z,w)|^2} d\sigma(z,w).
\]
For the moment, we study properties that hold in general. Recall
$\mathcal{I}_\mu = \mathbb{C}[z,w] \cap L^2(\mu)$.

\begin{definition}
  We say an element $p$ of $\mathbb{C}[z,w]$ is a \emph{divisor of the
    ideal} $\mathcal{I}_\mu$ if whenever $p q \in \mathcal{I}_\mu$,
  then $q \in \mathcal{I}_\mu$.
\end{definition}

Polynomials with no zeros on $\mathbb{T}^2$ are always divisors of
$\mathcal{I}_\mu$.  The following proposition presents some
restrictions on the factors of certain subspaces of polynomials
defined by $\mu$.

\begin{prop}\ 

\begin{enumerate}
\item   
\begin{enumerate}
\item If $p$ is a nonzero element of $\gkboxuperprt_\mu$ or
  $\gkboxperprt_\mu$, then $p$ is not divisible by a polynomial of the
  form $L(z,w) = z - z_0$ for $z_0 \in \mathbb{D}$. 
\item  If $p$ is a
  nonzero element of $\gkboxuperplt_\mu$ or $\gkboxperplt_\mu$ then
  $p$ is not divisible by any $L(z,w) = z-z_0$ when $z_0 \in
  \mathbb{C}\setminus \overline{\mathbb{D}}$. 
 \item In addition, if $z_0
  \in \mathbb{T}$, and $L(z,w) = z-z_0$ happens to be a divisor in
  $\mathcal{I}_\mu$, then nonzero elements of $\gkboxuperprt_\mu,
  \gkboxuperplt_\mu, \gkboxperprt_\mu, \gkboxperplt_\mu$ cannot have
  $L$ as a factor.
\end{enumerate}

\item 
\begin{enumerate}
\item If $p$ is a nonzero element of $\gkboxrperpup_\mu$ or
$\gkboxperpup_\mu$, then $p$ cannot have a factor of the form $J(z,w) =
w-w_0$ when $w_0 \in \mathbb{D}$.  

\item If $p$ is a nonzero element
of $\gkboxrperpdn_\mu$ or $\gkboxperpdn_\mu$, then $p$ cannot have a
factor of the form $J(z,w) = w - w_0$ when $w_0 \in \mathbb{C}
\setminus \overline{\mathbb{D}}$. 

\item In addition, if $w_0 \in \mathbb{T}$, and $J(z,w) = w-w_0$ happens to
be a divisor in $\mathcal{I}_\mu$, then nonzero elements of
$\gkboxrperpup_\mu, \gkboxrperpdn_\mu, \gkboxperpup_\mu,
\gkboxperpdn_\mu$ cannot have $J$ as a factor.
\end{enumerate}
\end{enumerate}
\end{prop}


\begin{proof} We prove item (1a).  Let $p \in \gkboxuperprt_\mu$ and
  suppose $p = gL$ for some $g \in \gkboxsm$ where $L(z,w) = z-z_0$
  with $|z_0|< 1$.  Since $L$ has no zeros on $\mathbb{T}^2$, $g = p/L
  \in L^2(\mu)$.  Then, $z_0 g(z,w) = z g(z,w)-p(z,w)$ and
\[
|z_0|^2 ||g||^2_{L^2(\mu)} = ||-p+z g||^2_{L^2(\mu)} =
||p||_{L^2(\mu)}^2 + ||z g||_{L^2(\mu)}^2 = ||p||_{L^2(\mu)}^2 +
||g||_{L^2(\mu)}^2.
\]
since $p \perp_\mu z g$.  Rearranging we arrive at
\[
||p||_{L^2(\mu)}^2 = (|z_0|^2-1) ||g||_{L^2(\mu)}^2 < 0,
\] 
a contradiction.
The proofs of the other statements are variations on the above idea.
\end{proof}

Curiously, slightly more complicated factors can be ruled out by a
similar argument.  For instance, if $|a| < 1$, then $P(z,w) = z^2 - a
w^3$ cannot be a factor of any polynomial in $\gkboxperprt_\mu$.  If
$|a|=1$ and $P$ is a divisor of $\mathcal{I}_\mu$ then the same
conclusion holds.

\begin{prop} \label{fullrankprop} Let $\{e_1, \dots, e_N \} \subset
  \mathbb{C}[z,w]$ be an orthonormal basis for $\gkboxrperpup_\mu$
  which we write vectorially as $\vec{E}(z,w) = (e_1(z,w), \dots,
  e_N(z,w))^t$ which we in turn write as
\[
\vec{E}(z,w) = E(w) \mathbf{\Lambda}_n(z)
\]
where $E(w)$ is an $(N\times n)$-matrix valued polynomial in $w$ of
degree at most $m$.  Then, $E(w_0)$ has rank $N$ for all $w_0 \in
\mathbb{D}$ and for all $w_0 \in \mathbb{T}$ with the property that
$L(z,w) = w - w_0$ is a divisor of $\mathcal{I}_\mu$.  The same
results hold for $\gkboxuperprt_\mu$ with the roles of $z$ and $w$
switched.
\end{prop}

\begin{proof} First, we claim $\dim \gkboxrperpup_\mu := N \leq n$.
  Given $n+1$ polynomials in $\gkboxrperpup_\mu$, some linear
  combination of them will be a multiple of $w$ (since the degree in
  $z$ is at most $n-1$); such a combination would be orthogonal to
  itself (by definition of $\gkboxrperpup_\mu$) and therefore zero;
  and hence any $n+1$ polynomials in $\gkboxrperpup_\mu$ are
  dependent.  So, $\dim \gkboxrperpup_\mu \leq n$.

  Next, suppose $E(w_0)$ has rank less than $N$ at some point $w_0 \in
  \mathbb{C}$.  Since $E(w_0)$ is $N\times n$ and $N\leq n$ there must
  be a nonzero vector $\mathbf{v} \in \mathbb{C}^N$ such that $\mathbf{v}^t
  E(w_0) = \mathbf{0}^t$; i.e. the following (necessarily nonzero)
  polynomial
\[ 
q(z,w) = \mathbf{v}^t E(w) \mathbf{\Lambda}_n(z) = \mathbf{v}^t \vec{E}(z,w)
\]
is in $\gkboxrperpup_\mu$ and vanishes on the set $\{w=w_0\}$. By the
previous proposition this can only happen if $w_0 \notin \mathbb{D}$
and if it happens that $w_0\in \mathbb{T}$, $w-w_0$ cannot be a
divisor of $\mathcal{I}_\mu$.  So, $E(w_0)$ has full rank $N$ everywhere
in $\mathbb{D}$ and at all points $w_0\in \mathbb{T}$ for which
$w-w_0$ is a divisor of $\mathcal{I}_\mu$.
\end{proof}

Continuing our previous aside, we can also say that $\mathbf{E} \in
\mathbb{C}^{N}[z,w]$ as above when restricted to the variety $\{z^2 -
aw^3=0\}$ (here $|a|<1$) does not sit inside any proper subspace of
$\mathbb{C}^N$.  

\begin{remark} The main ideas of the previous two propositions
  appeared in the appendix of Knese \cite{gK08b} in a less detailed
  form.
\end{remark}

\begin{definition} \label{t2symmetric}
A polynomial $p \in \mathbb{C}[z,w]$ is $\mathbb{T}^2$-symmetric if it
equals a unimodular constant $\mu$ times its reflection:
\[
p(z,w) = \mu \refl{p}(z,w) = \mu z^j w^k \overline{p(1/\bar{z},
  1/\bar{w})};
\]
here $p$ has degree exactly $(j,k)$. 
\end{definition}

\begin{prop} \label{gcdprop} Let $P$ be the greatest common divisor of
  $\gkbox_\mu$.  Then, every factor of $P$ is $\mathbb{T}^2$-symmetric
  and the zero set of every factor of $P$ intersects $\mathbb{T}^2$.
\end{prop}

\begin{proof}
  The greatest common divisor $P$ is necessarily
  $\mathbb{T}^2$-symmetric (basically since the set $\gkbox_\mu$ is).
  Let $q$ be an irreducible factor of $P$ and let $j$ be the highest
  power such that $q^j$ divides $P$.  Suppose $q$ is not a multiple of
  $\refl{q}$. Then $q^j\refl{q}^j$ divides $P$. Let $p$ be an element
  of $\gkbox_\mu$ divisible by the maximal number of factors of $q$;
  i.e. $q^k$ divides $p$ and no nonzero element of $\gkbox_\mu$ is
  divisible by $q^{k+1}$.  Since $\refl{q}^j$ divides $p$ we may write
  $p = q^k \refl{q}^j g$ for some $g \in \mathbb{C}[z,w]$.  Since $|q|
  = |\refl{q}|$ on $\mathbb{T}^2$, it follows that $p$ being in
  $L^2(\mu)$ implies $q^{k+j} g \in L^2(\mu)$.  In particular,
  $q^{k+j} g \in \gkbox_\mu$ contradicting the maximality property of
  $p$ and $k$.  Hence, $q$ must be $\mathbb{T}^2$-symmetric.

The zero set of every factor $q$ of $P$ must intersect $\mathbb{T}^2$
since otherwise $qg \in L^2(\mu)$ implies $g \in L^2(\mu)$ for any
$g\in \mathbb{C}[z,w]$.   
\end{proof}

\begin{question} \label{gcdquestion} Is $P$ toral? i.e. does the zero
  set of every factor of $P$ intersect $\mathbb{T}^2$ on an infinite
  set?
\end{question}

This question is made more difficult by the fact that there exist
irreducible, atoral, $\mathbb{T}^2$-symmetric polynomials:
\[
p(z,w) = (3z+1)w^2 - (z+3)(3z+1)w + z(z+3)
\]
is such a polynomial taken from Agler-M\mcc Carthy-Stankus\cite{AMS08}.

\section{What happens when $\gkboxurperpdn_\mu = \gkboxrperpdn_\mu$?} \label{sec:ao}

In this section we explore the implications of the assumption
$\gkboxurperpdn_\mu = \gkboxrperpdn_\mu$.  The most immediate
consequence of this assumption is the formula
\[
\begin{aligned}
  q(z,w)& \overline{q(Z,W)} - \refl{q}(z,w) \overline{\refl{q}(Z,W)} \\
  =& (1-z \bar{Z}) K\gkboxrperpdn_\mu((z,w),(Z,W)) + (1-w \bar{W})
  K\gkboxuperprt_\mu((z,w),(Z,W))\\
\end{aligned}
\]
where $q$ is any unit norm polynomial in $\gkboxllperp_\mu$.  This is
just Theorem \ref{rearrangethm} with $\epsilon = 0$ (as mentioned
there, $\gkboxurperpdn_\mu = \gkboxrperpdn_\mu$ implies $\epsilon
=0$).  Evaluating on the diagonal $(z,w) = (Z,W)$ we have
\begin{align} \label{zerosineq}
|q(z,w)|^2 \geq& |q(z,w)|^2 - |\refl{q}(z,w)|^2 \\
 = &(1-|z|^2)K\gkboxrperpdn_\mu((z,w),(z,w)) \nonumber \\
 & + (1-|w|^2)K\gkboxuperprt_\mu ((z,w),(z,w)) \geq 0 \nonumber
\end{align}
for all $(z,w) \in \cbidisk$.  If we scrutinize this inequality, we
can prove something quite strong.

\begin{prop} \label{zerosprop} Suppose $\gkboxurperpdn_\mu =
  \gkboxrperpdn_\mu$ and let $q$ be any unit norm polynomial in
  $\gkboxllperp_\mu$.  If $q(z_0,w_0) = 0$ for some $(z_0,w_0) \in
  \cbidisk$, then every element of $\gkbox_\mu$ vanishes at
  $(z_0,w_0)$.
\end{prop}

\begin{proof}
Two formulas will be useful in what follows:
\begin{equation} \label{whatfollows1}
K\gkboxuperprt_\mu - K\gkboxuperplt_\mu = (1-|z|^2) K\gkboxsm_\mu
\end{equation}
and
\begin{equation} \label{whatfollows2}
K\gkbox_\mu = K\gkboxsm_\mu + K\gkboxuperplt_\mu + K\gkboxurperpdn_\mu
+ \refl{q}\overline{\refl{q}}
\end{equation}
where every reproducing kernel is evaluated on the diagonal $(z,w) =
(Z,W)$. The first formula follows from the fact that
\[
\gkboxu_\mu = \gkboxuperprt_\mu \oplus (z\gkboxsm_\mu) =
\gkboxuperplt_\mu \oplus \gkboxsm_\mu
\]
and the second follows from the fact that
\[
\gkbox_\mu = \gkboxsm_\mu\oplus \gkboxuperplt_\mu \oplus
\gkboxurperpdn_\mu \oplus \gkboxurperp_\mu.
\]

First, suppose $(z_0,w_0) \in \mathbb{D}^2$. We write $v =(z_0,w_0)$
for short. From \eqref{zerosineq}, it is immediate that $q(v) = 0$
implies
\begin{equation} \label{somezeros}
\refl{q}(v) = K\gkboxrperpdn_\mu(v,v) =
K\gkboxuperprt_\mu(v,v) = 0.
\end{equation}
This is enough to force $K\gkbox_\mu(v, v) = 0$ by formulas
\eqref{whatfollows1} and \eqref{whatfollows2}.  Indeed,
$K\gkboxuperprt_\mu(v,v) = 0$ implies $K\gkboxuperplt_\mu(v,v) =
K\gkboxsm_\mu(v,v) = 0$ by \eqref{whatfollows1} (using the fact that
reproducing kernels are non-negative on the diagonal).  Then,
\eqref{whatfollows2} implies $K\gkbox_\mu(v,v) = 0$ since
$K\gkboxurperpdn_\mu = K\gkboxrperpdn_\mu$ by assumption.  If
$K\gkbox_\mu (v,v) = 0$ then every element of $\gkbox_\mu$ must vanish
at $v$.

To prove the claim for $v = (z_0,w_0) \in \cbidisk\setminus \mathbb{D}^2$,
notice that the left hand side of \eqref{zerosineq} vanishes to order
at least two at $v$, and the terms $(1-|z|^2)$ and $(1-|w|^2)$
can vanish to order at most one.  This again implies \eqref{somezeros}
and by a similar argument $K\gkbox_\mu(v,v) = 0$.

This proves every element of $\gkbox_\mu$ vanishes at a zero of $q$ in
$\cbidisk$.  
\end{proof}

\begin{remark} \label{stableremark} 
If $\mu$ is a finite measure, then $1 \in \gkbox_\mu$ and this implies
$q$ has no zeros on the closed bidisk. Hence, this proves stability in
the case of probability measures, as in Geronimo-Woerdeman \cite{GW04}
and Knese \cite{gK08a}.
\end{remark}

\begin{corollary} \label{factorcorollary} Suppose $\gkboxurperpdn_\mu
  = \gkboxrperpdn_\mu$ and let $q$ be any unit norm polynomial in
  $\gkboxllperp_\mu$.  Then, $q$ can be factored into $q = q_1
  q_2$ where
\begin{itemize} 
\item $q_1$ divides every element of $\gkbox_\mu$;
\item every irreducible factor of $q_1$: is $\mathbb{T}^2$-symmetric,
  has infinitely many zeros in $\cbidisk$, and vanishes somewhere on
  $\mathbb{T}^2$; and
\item $q_2$ has no zeros in $\cbidisk \setminus \mathbb{T}^2$ and
  finitely many zeros in $\mathbb{T}^2$.
\end{itemize}
\end{corollary}

\begin{proof}
It is clear $q$ may be factored into the form $q = q_1 q_2$ where
every irreducible factor of $q_1$ has infinitely many zeros in
$\cbidisk$ and $q_2$ has finitely many zeros in $\cbidisk$ (we of
course allow for the case where $q_1$ or $q_2$ is a constant). 

Suppose $f$ is an irreducible factor of $q$ possessing infinitely many
zeros in $\cbidisk$; i.e. a factor of $q_1$.  By Proposition
\ref{zerosprop}, every element of $\gkbox_\mu$ has infinitely many
zeros in common with $f$ and hence $f$ divides every element of
$\gkbox_\mu$. This implies $f$ can be divided out of both sides of the
inequality \eqref{zerosineq} and using the resulting inequality one
can then show that if $f$ occurs in the factorization of $q$ with
multiplicity, it then divides every element of $\gkbox_\mu$ with the
same multiplicity.  This implies $q_1$ divides every element of
$\gkbox_\mu$.  By Proposition \ref{gcdprop}, any such $f$ necessarily
is $\mathbb{T}^2$-symmetric and vanishes somewhere on $\mathbb{T}^2$.
This proves the first two items in the statement of the corollary.

Finally, if $q_2$ has finitely many zeros in $\cbidisk$, $q_2$ can
have no zeros in the bidisk.  By Lemma \ref{zeroslemma}, $q_2$ can
have no zeros on the sides: $\mathbb{D}\times \mathbb{T}$ and
$\mathbb{T}\times\mathbb{D}$.  This proves the third item.
\end{proof}

Since the factor $q_1$ in the above corollary divides every element of
$\gkbox_\mu$, the study of $\mu$ and $\gkbox_\mu$ can be separated
into the study of $q_1$ and the study of $|q_1|^2d\mu$ and the set
$\gkbox_\mu/q_1$ (which is nothing more than all $p \in L^2(|q_1|^2
d\mu)$ of degree less than or equal to $(n-n_1,m-m_1)$, where
$(n_1,m_1)$ is the degree of $q_1$).  Indeed, the map sending
\[
f \in \gkbox_\mu \mapsto f/q_1 \in \gkbox_\mu/q_1
\]
is an isometry (using the inner product of $L^2(\mu)$ on the left and
the inner product of $L^2(|q_1|^2 d\mu)$ on the right).  Although this
is a somewhat trivial observation, we now feel justified in making the
assumption that $\gkboxurperp_\mu$ and $\gkboxllperp_\mu$ have no
common factor.  This is equivalent to saying $q$ and $\refl{q}$ have
no common factor, which is equivalent to saying $q_1$ is a constant.
In this case the following proposition is immediate, since the
assumption implies $q = q_2$ in Corollary \ref{factorcorollary}.

\begin{prop} If $\gkboxurperpdn_\mu = \gkboxrperpdn_\mu$ and if
  $\gkboxurperp_\mu$ and $\gkboxllperp_\mu$ are one-dimensional and
  have no factor in common, then any $q \in \gkboxllperp_\mu$ has no
  zeros on $\overline{\mathbb{D}}^2\setminus \mathbb{T}^2$ and
  finitely many zeros on $\mathbb{T}^2$.
\end{prop}

\begin{lemma} \label{dimensionlemma} Suppose $\gkboxurperp_\mu$ is one
  dimensional and has no factor in common with $\gkboxllperp_\mu$, and
  suppose $\gkboxurperpdn_\mu = \gkboxrperpdn_\mu$.
  Then,
\[
\dim \gkboxrperpup_\mu = n \text{ and } \dim \gkboxuperplt_\mu = m.
\]
\end{lemma}

\begin{proof}  
  Let $h$ be a unit norm polynomial in $\gkboxurperp_\mu$.  This
  polynomial $h$ necessarily has degree exactly $(n,m)$, otherwise it
  would be orthogonal to itself.  Set $q = \refl{h}$, where the
  reflection is performed at the $(n,m)$ level.  By Theorem
  \ref{rearrangethm} with $\epsilon =0$,
\[
\begin{aligned}
&q(z,w)\overline{q(Z,W)} - \refl{q}(z,w) \overline{\refl{q}(Z,W)}  \\
& = (1-z\bar{Z}) K\gkboxrperpup_\mu ((z,w),(Z,W)) + (1-w\bar{W})
K\gkboxuperplt_\mu ((z,w),(Z,W)).
\end{aligned}
\]

Let $d_1 = \dim \gkboxrperpup_\mu$ and $d_2 = \dim \gkboxuperplt_\mu$;
let $e_1, \dots, e_{d_1}$ be an orthonormal basis for
$\gkboxrperpup_\mu$ and $f_1, \dots, f_{d_2}$ be an orthonormal basis
for $\gkboxuperplt_\mu$.  We write these vectorially as
\[
\vec{E}(z,w) = \begin{pmatrix} e_1(z,w) \\ \vdots \\
  e_{d_1}(z,w) \end{pmatrix} \text{ and } \vec{F}(z,w) = \begin{pmatrix}
  f_1(z,w) \\ \vdots \\ f_{d_2} (z,w) \end{pmatrix}
\]
and then the formula above becomes
\[
\begin{aligned}
  &q(z,w)\overline{q(Z,W)} - \refl{q}(z,w) \overline{\refl{q}(Z,W)}  \\
  & = (1-z\bar{Z}) \ip{\mathbf{E}(z,w)}{\mathbf{E}(Z,W)} +
  (1-w\bar{W}) \ip{\mathbf{F}(z,w)}{\mathbf{F}(Z,W)}.
\end{aligned}
\]
Upon rearranging we have
\[
\begin{aligned}
q(z,w)\overline{q(Z,W)} &+ z\bar{Z}
\ip{\mathbf{E}(z,w)}{\mathbf{E}(Z,W)} + w\bar{W}
\ip{\mathbf{F}(z,w)}{\mathbf{F}(Z,W)} \\
=\refl{q}(z,w) \overline{\refl{q}(Z,W)}
&+\ip{\mathbf{E}(z,w)}{\mathbf{E}(Z,W)}
+\ip{\mathbf{F}(z,w)}{\mathbf{F}(Z,W)}
\end{aligned}
\]

The map which sends 
\[
\begin{pmatrix} q(z,w) \\ z\vec{E}(z,w) \\ w
  \vec{F}(z,w) \end{pmatrix} \mapsto \begin{pmatrix} \refl{q}(z,w) \\
  \vec{E}(z,w) \\ \vec{F}(z,w) \end{pmatrix}
\]
for each $(z,w) \in \mathbb{C}^2$ defines a unitary on the span of the
elements in $\mathbb{C}^{1+d_1+d_2}$ of the form on the left to the
span of the elements in $\mathbb{C}^{1+d_1+d_2}$ of the form on the
right, which can be extended to a $(1+d_1+d_2)\times (1+d_1+d_2)$
unitary matrix $U$. We write $U$ in block form as
\[
U = \begin{matrix} & \begin{matrix} \mathbb{C} &
    \mathbb{C}^{d_1+d_2} \end{matrix} \\
\begin{matrix} \mathbb{C} \\ \mathbb{C}^{d_1+d_2} \end{matrix}
& \begin{pmatrix} A & B \\ C & D \end{pmatrix} \end{matrix}
\]
We also define a $\mathbb{C}^{d_1+d_2}$-valued polynomial
$\vec{G}$ by
\[
\vec{G}(z,w) := \begin{pmatrix} \vec{E}(z,w) \\
  \vec{F}(z,w) \end{pmatrix}
\]
and define the $(d_1+d_2)\times (d_1+d_2)$ diagonal matrix
\[
\Delta(z,w) := \begin{pmatrix} zI_{d_1} & 0 \\ 0 & w
  I_{d_2} \end{pmatrix}.
\]

Then,
\[
\begin{aligned}
A q(z,w) +B \Delta (z,w) \vec{G}(z,w) &= \refl{q}(z,w) \\
C q(z,w) +D \Delta (z,w) \vec{G}(z,w) &= \vec{G}(z,w)
\end{aligned}
\]

The latter formula implies
\[
\vec{G}(z,w) = q(z,w) (I - D \Delta(z,w))^{-1} C 
\]
and in turn the former formula implies
\[
A+B \Delta(z,w) (I-D \Delta(z,w))^{-1} C =
\frac{\refl{q}(z,w)}{q(z,w)}.
\]

Since $\refl{q}/q$ is already in reduced terms we must have $d_1\geq
n$ and $d_2 \geq m$.  We already know $d_1 \leq n$ and $d_2 \leq m$
(see Proposition \ref{fullrankprop}). Therefore, $n = \dim
\gkboxrperpup_\mu$ and $m= \dim \gkboxuperprt_\mu$, and the result
follows.
\end{proof}

\begin{theorem}[``Spectral Matching''] \label{spectralmatching} Let $\mu$ and $\rho$ be two
  positive Borel measures satisfying
  \begin{equation} \label{spectralassumption} \gkboxurperpdn_{\mu} =
    \gkboxrperpdn_{\mu} \qquad \gkboxurperpdn_{\rho} =
    \gkboxrperpdn_{\rho}.
\end{equation}
Suppose $\gkboxllperp_{\mu} = \gkboxllperp_\rho \ne \{0\}$ and let $q
\in \gkboxllperp_\mu$.  Assume $q$ and $\refl{q}$ have no common
factor. Then, $\gkbox_\mu = \gkbox_\rho$ and the inner products
$\ip{}{}_\mu$ and $\ip{}{}_\rho$ agree up to a constant multiple on
$\gkbox_\mu$; i.e.
\[
\frac{1}{||q||^2_{L^2(\mu)}} \ip{f}{g}_\mu =
\frac{1}{||q||^2_{L^2(\rho)}} \ip{f}{g}_\rho
\]
for all $f,g \in \gkbox_\mu$.  In other words,
\[
\frac{1}{||q||^2_{L^2(\mu)}} K\gkbox_\mu =
\frac{1}{||q||^2_{L^2(\rho)}} K\gkbox_\rho.
\]
\end{theorem}

\begin{proof}
  We may renormalize $\mu$ and $\rho$ so that $1 = ||q||_{L^2(\mu)} =
  ||q||_{L^2(\rho)}$.

  By choosing orthonormal bases for the $n$-dimensional subspaces (by
  Lemma \ref{dimensionlemma}) $\gkboxrperpup_\mu$ and
  $\gkboxrperpup_\rho$ we may write
\[
K\gkboxrperpup_\mu ((z,w),(Z,W)) =
\ip{\mathbf{E}_\mu(z,w)}{\mathbf{E}_\mu(Z,W)}
\]
\[
K\gkboxrperpup_\rho ((z,w),(Z,W)) =
\ip{\mathbf{E}_\rho(z,w)}{\mathbf{E}_\rho(Z,W)}
\]
for $\mathbf{E}_\mu, \mathbf{E}_\rho \in \mathbb{C}^n[z,w]$. 

Likewise, we may write the $m$-dimensional subspaces
$\gkboxuperplt_\mu$ and $\gkboxuperplt_\rho$ as
\[
K\gkboxuperplt_\mu ((z,w),(Z,W)) =
\ip{\mathbf{F}_\mu(z,w)}{\mathbf{F}_\mu(Z,W)} 
\]
\[
K\gkboxuperplt_\rho ((z,w),(Z,W)) =
\ip{\mathbf{F}_\rho(z,w)}{\mathbf{F}_\rho(Z,W)}
\]
where $\mathbf{F}_\mu, \mathbf{F}_\rho \in \mathbb{C}^m[z,w]$.

By Proposition \ref{fullrankprop}, both $\mathbf{E}_\mu,
\mathbf{F}_\mu$ and $\mathbf{E}_\rho, \mathbf{F}_\rho$ satisfy the
hypotheses of Lemma \ref{uniquenesslemma} (in place of $\mathbf{E},
\mathbf{F}$ and $\mathbf{\tilde{E}}, \mathbf{\tilde{F}}$), since by
Theorem \ref{rearrangethm}, we have
\[
\begin{aligned}
& (1-z \bar{Z})K\gkboxrperpup_\mu ((z,w),(Z,W)) + (1-w\bar{W}) K
\gkboxuperplt_\mu ((z,w),(Z,W)) \\
&= (1-z \bar{Z})K\gkboxrperpup_\rho ((z,w),(Z,W)) + (1-w\bar{W}) K
  \gkboxuperplt_\rho ((z,w),(Z,W))
\end{aligned}
\]

Therefore, $\mathbf{E}_\mu$ is a unitary multiple of $\mathbf{E}_\rho$
and $\mathbf{F}_\mu$ is a unitary multiple of $\mathbf{F}_\rho$.  In
other words,
\[
K\gkboxuperplt_\mu ((z,w),(Z,W)) = K\gkboxuperplt_\rho ((z,w),(Z,W))
\]
\begin{equation} \label{spectraleqn1}
K\gkboxrperpup_\mu ((z,w),(Z,W)) = K\gkboxrperpup_\rho ((z,w),(Z,W))
\end{equation}

Now we will see that this is all that is needed to reassemble the two
inner products on $\gkbox_\mu$ or $\gkbox_\rho$.

By reflection
\[
K\gkboxuperprt_\mu ((z,w),(Z,W)) = K\gkboxuperprt_\rho ((z,w),(Z,W))
\]
and by the formulas (which hold for both $\mu$ and $\rho$)
\[
K\gkboxuperprt_\mu - K\gkboxuperplt_\mu = (1-|z|^2) K\gkboxsm_\mu
\]
and
\[
K\gkbox_\mu = K\gkboxsm_\mu + K\gkboxuperplt_\mu + K\gkboxurperpdn_\mu
+ \refl{q}\overline{\refl{q}}
\]
where every reproducing kernel is evaluated on the diagonal $(z,w) =
(Z,W)$, we see that
\[
K\gkbox_\mu = K\gkbox_\rho.
\]
(This is similar to the argument in the proof of Proposition
\ref{zerosprop}.)
\end{proof}

\section{Bernstein-Szeg\H{o} measures} \label{sec:BernSzeg}

Converse to the previous section, we now study Bernstein-Szeg\H{o}
measures, which will be shown to satisfy $\gkboxurperpdn_\mu =
\gkboxrperpdn_\mu$.  Bernstein-Szeg\H{o} measures are measures on
$\mathbb{T}^2$ of the form
\[
d\mu = \frac{1}{|q(z,w)|^2} d\sigma(z,w)
\]
where $q \in \mathbb{C}[z,w]$ has no zeros on $\mathbb{D}^2$. (Recall
$d\sigma$ is normalized Lebesgue measure on $\mathbb{T}^2$.)

The following proposition looks innocuous, but it addresses the main
technical difficulty \emph{not present} in the case of polynomials
with no zeros on the entire \emph{closed} bidisk.  Note this
proposition does not require the polynomial to have finitely many
zeros on $\mathbb{T}^2$.

\begin{prop} \label{qrelations} Let $q \in \mathbb{C}[z,w]$ have
  degree at most $(n,m)$ and no zeros on $\mathbb{D}^2$.  Define a
  measure on $\mathbb{T}^2$ by
\[
d\mu = \frac{1}{|q(z,w)|^2} d\sigma(z,w)
\]
Then, $q \in \gkboxllperp_\mu$ and more generally
\[
q \perp_\mu \{f \in L^2(\mu) : \hat{f}(j,k) =0 \text{ for } k < 0
\text{ and for } k=0 \text{ and } j \leq 0 \}.
\]  
\end{prop} 

\begin{proof} Let $f \in L^2(\mu)$ satisfy 
\[
\hat{f}(j,k) =0 \text{ for } k < 0 \text{ and for } k=0 \text{ and } j
\leq 0.
\]

It is necessarily true that $f\in L^2(\mathbb{T}^2)$.  For almost
every $z \in \mathbb{T}$, the function $f_{(z)} (w) = f(z,w)$ is in
$L^2(\mathbb{T})$ and since $\hat{f}(j,k) = 0$ for $k<0$, $f_{(z)}$ is
actually in $H^2(\mathbb{T})$ for almost every $z \in \mathbb{T}$.

This implies the function (of $w$)
\[
g_{(z)} (w) := \frac{f(z,w)}{q(z,w)}
\]
is in the Smirnov class $N^+$ (which consists of all ratios of bounded
analytic functions with outer denominator; see Duren \cite{pD70},
section 2.5) for almost every $z \in \mathbb{T}$: $q(z,\cdot)$ has no
zeros in the disk for all but finitely many $z \in \mathbb{T}$ (by
Lemma \ref{zeroslemma}) and is therefore \emph{outer} for almost every
$z \in \mathbb{T}$.  Since $f \in L^2(\mu)$, Fubini's theorem says that
for almost every $z \in \mathbb{T}$, we have $g_{(z)} \in
L^2(\mathbb{T})$. By Theorem 2.11 in Duren \cite{pD70}, $N^+ \cap
L^2(\mathbb{T}) = H^2(\mathbb{T})$, and therefore $g_{(z)} \in
H^2(\mathbb{T})$ for almost every $z \in \mathbb{T}$.

Owing to the fact that $g_{(z)}$ is orthogonal to $w^j$ for $j<0$,
\[
\begin{aligned}
  f(z,0) = \int_{\mathbb{T}} f(z,w) \frac{dw}{2\pi i w} &=
  \int_{\mathbb{T}} \frac{f(z,w)}{q(z,w)} q(z,w) \frac{dw}{2\pi i w} \\
  &= \int_{\mathbb{T}} \frac{f(z,w)}{q(z,w)} q(z,0) \frac{dw}{2\pi i
    w}
\end{aligned}
\]
for almost every $z \in \mathbb{T}$, and so
\[
\int_{\mathbb{T}^2} \frac{f(z,w)}{q(z,w)} \frac{dw}{2\pi i
    w} \frac{dz}{2 \pi i z} = \int_{\mathbb{T}} \frac{f(z,0)}{q(z,0)}
  \frac{dz}{2 \pi i z}. 
\]
Now, the function defined by $h(z) = f(z,0)/q(z,0)$ is in
$L^2(\mathbb{T})$ by Fubini's theorem.  Also, $h$ is in the Smirnov
class $N^+$ because $f(\cdot,0)$ is in $H^2(\mathbb{T})$ (by the
assumption that $\hat{f}(j,0) = 0$ for $j \leq 0$) and $q(\cdot, 0)$
is outer since $q(z,0)$ has no zeros in the disk.  Therefore, $h$ is
in $H^2(\mathbb{T})$. Thus, we may conclude
\[
\int_{\mathbb{T}^2} \frac{f(z,w)}{q(z,w)} \frac{dw}{2\pi i
    w} \frac{dz}{2 \pi i z} = \int_{\mathbb{T}} \frac{f(z,0)}{q(z,0)}
  \frac{dz}{2 \pi i z} = \frac{f(0,0)}{q(0,0)} = 0
\]
since $\hat{f}(0,0) = 0$.

Since 
\[
\ip{f}{q}_\mu = \int_{\mathbb{T}^2} \frac{f(z,w)
  \overline{q(z,w)}}{|q(z,w)|^2} d\sigma(z,w) = 
\int_{\mathbb{T}^2} \frac{f(z,w)}{q(z,w)} d\sigma(z,w)
\]
we have shown $\ip{f}{q}_\mu = 0$, or in other words $f \perp_\mu q$.

\end{proof}

From here, the proofs follow the stable case, as in Knese
\cite{gK08a}, with some minor changes.

\begin{corollary} \label{perpcorollary} If $f \in L^2(\mu) \cap
  H^2(\mathbb{T}^2)$ and
\[
\hat{f}(j,k) = 0 \text{ for } k > m \text{ and for } k=m \text{ and }
j\geq n ,
\]
then $\ip{f}{\refl{q} g}_\mu = 0$ for any $g \in
H^{\infty}(\mathbb{T}^2)$.
\end{corollary}

\begin{proof}
Notice that $\ip{\refl{q}g}{f}_\mu = \ip{\bar{f} g z^n w^m}{q}_\mu$.
Also, notice that $\bar{f} g z^n w^m$ satisfies the hypotheses of the
previous proposition (it helps to draw a picture of the frequency
support of $f$ and $\bar{f} g z^n w^m$).  Therefore, $\ip{f}{\refl{q}
g}_\mu = 0$.
\end{proof}

\begin{lemma} \label{Llemma}
Define
\begin{align}
L_{(Z,W)} (z,w) &= L((z,w),(Z,W)) \nonumber \\
&= (z\bar{Z})^n
\frac{q(z,w)\overline{q(1/\bar{z}, W)} - \refl{q}(z,w)
  \overline{\refl{q}(1/\bar{z}, W)}}{(1-z\bar{Z})(1-w\bar{W})}
\label{def:L} 
\end{align}

Suppose $f \in L^2(\mu)\cap H^2(\mathbb{T}^2)$ with
\[
\hat{f}(j,k) = 0 \text{ for } k > m \text{ and for } k=m \text{ and }
j \geq n.
\]
Then, for $(Z,W)\in \mathbb{D}^2$
\[
\sum_{k=0}^{m-1} \sum_{j=n}^{\infty} \hat{f}(j,k) Z^j W^k =
\ip{f}{L_{(Z,W)}}_\mu
\]
\end{lemma}

\begin{proof}  By Corollary \ref{perpcorollary}, $f$ is orthogonal to
  the function
\[
G_{(Z,W)} (z,w) = \frac{\refl{q}(z,w) z^n
  \overline{\refl{q}(1/\bar{z}, W)}}{(1-z\bar{Z})(1-w\bar{W})}
\]
for each $(Z,W) \in \mathbb{D}^2$.

Therefore, 
\begin{align}
  \ip{f}{L_{(Z,W)}}_{\mu} &= \int_{\mathbb{T}^2} \frac{f(z,w)
    \overline{q(z,w)} q(z,W) (\bar{z}Z)^n}{(1-\bar{z}Z)(1-\bar{w}W)
    |q(z,w)|^2} \frac{dw dz}{(2\pi i)^2 zw} \nonumber \\
  &= \int_{\mathbb{T}} \int_{\mathbb{T}} \frac{f(z,w) q(z,W)
    (\bar{z}Z)^n}{(1-\bar{z} Z)(w-W)
    q(z,w)} \frac{dw}{2\pi i} \frac{dz}{2 \pi i z} \label{therefore1} \\
  &= \int_{\mathbb{T}} \frac{f(z,W)}{q(z,W)} q(z,W)
  \frac{(\bar{z}Z)^n}{(1-\bar{z}Z)} \frac{dz}{2\pi i z}
    \label{therefore2} \\
&= \sum_{j=n}^{\infty} \sum_{k=0}^{m-1} \hat{f}(j,k) Z^j
    W^k. \label{therefore3} 
\end{align}
Going from \eqref{therefore1} to \eqref{therefore2} is an application
of the Cauchy integral formula and going from \eqref{therefore2} to
\eqref{therefore3} involves cancellation and another application of
the Cauchy integral formula.
\end{proof}

\begin{theorem} \label{BernSzegthm} Let $q$ be a nonzero polynomial of
  degree at most $(n,m)$ with no zeros on $\mathbb{D}^2$.  Define a
  measure on $\mathbb{T}^2$ by
\[
d\mu = \frac{1}{|q(z,w)|^2} d\sigma(z,w).
\]
Then, 
\[
\gkboxurperpdn_\mu = \gkboxrperpdn_\mu.
\]
\end{theorem}

\begin{proof} Let
\[
\mathbf{HS} = \{f \in L^2(\mu)\cap H^2(\mathbb{T}^2) : \hat{f}(j,k) =
0 \text{ for } k \geq m \}
\]
($\mathbf{HS} = $ ``half strip'') and let
\begin{multline*}
\mathbf{NHS} = \\ \{f \in L^2(\mu)\cap H^2(\mathbb{T}^2) : \hat{f}(j,k) =
0 \text{ for } k > m \text{ and when } k=m \text{ and } j\geq n\}
\end{multline*}
($\mathbf{NHS} = $ ``notched half strip'').

We claim that $\mathbf{NHS} \ominus_\mu \mathbf{HS} =
\gkboxrperpdn_\mu$. To prove $\mathbf{NHS} \ominus_\mu \mathbf{HS}
\subset \gkboxrperpdn_\mu$, notice that $L_{(Z,W)}$ from Lemma
\ref{Llemma} is in $\mathbf{HS}$ since the numerator of $L_{(Z,W)}$
vanishes when $w = 1/\bar{W}$, and hence $L_{(Z,W)}$ is a polynomial
of degree at most $m-1$ in $w$.  So, if $f \in \mathbf{NHS}
\ominus_\mu \mathbf{HS}$, then
\[
0 = \ip{f}{L_{(Z,W)}}_\mu = \sum_{j=n}^{\infty} \sum_{k=0}^{m-1}
  \hat{f}(j,k) Z^j W^k
\]
which means $f \in \gkboxr_\mu$ and therefore $f \in
\gkboxrperpdn_\mu$.  This proves $\mathbf{NHS} \ominus_\mu \mathbf{HS}
\subset \gkboxrperpdn_\mu$.

To prove $\gkboxrperpdn_\mu \subset \mathbf{NHS} \ominus_\mu
\mathbf{HS}$, let $P_{\mathbf{HS}}:L^2(\mu) \to \mathbf{HS}$ denote
the orthogonal projection onto $\mathbf{HS}$, a necessarily closed
subspace of $L^2(\mu)$ (the topology on $L^2(\mu)$ is finer than the
topology on $L^2(\mathbb{T}^2)$).  If $f \in \gkboxrperpdn_\mu$ then
\[
f - P_{\mathbf{HS}}f \in \mathbf{NHS}\ominus_\mu \mathbf{HS} \subset
\gkboxrperpdn_\mu
\]
and this implies $P_{\mathbf{HS}} f \in \gkboxrperpdn_{\mu} \cap
\mathbf{HS} = \{0\}$.  Hence, $P_{\mathbf{HS}} f = 0$ which means $f
\perp_\mu \mathbf{HS}$.  In other words, $f \in
\mathbf{NHS}\ominus_\mu \mathbf{HS}$.  Hence, $\mathbf{NHS}\ominus_\mu
\mathbf{HS} = \gkboxrperpdn_\mu$.

Now, since $\gkboxrperpdn_\mu \subset \mathbf{NHS} \ominus_\mu
\mathbf{HS}$, it follows that $\gkboxrperpdn_\mu \subset
\gkboxurperpdn_\mu$.  A similar argument to the above (using the
projection $P_{\mathbf{HS}}$) proves $\gkboxurperpdn_\mu \subset
\mathbf{NHS} \ominus_\mu \mathbf{HS} = \gkboxrperpdn_\mu$.  This
implies $\gkboxrperpdn_\mu = \gkboxurperpdn_\mu$.
\end{proof}

\begin{corollary} \label{bernszegsos} Let $q$ be a nonzero polynomial
  of degree at most $(n,m)$ with no zeros on $\mathbb{D}^2$.  Define a
  measure on $\mathbb{T}^2$ by
\[
d\mu = \frac{1}{|q(z,w)|^2} d\sigma(z,w).
\]
Then, 
\[
\begin{aligned}
  q(z,w) &\overline{q(Z,W)} - \refl{q}(z,w) \overline{\refl{q}(Z,W)}
  \\
  =& (1-z\bar{Z}) K\gkboxrperpup_\mu ((z,w),(Z,W))
  + (1-w\bar{W}) K\gkboxuperplt_\mu ((z,w),(Z,W)). \\
\end{aligned}
\]
\end{corollary}

\begin{proof} Proposition \ref{qrelations} says $q \in
  \gkboxllperp_\mu$ and Theorem \ref{BernSzegthm} says
  $\gkboxurperpdn_\mu = \gkboxrperpdn_\mu$.  Since $||q||_{L^2(\mu)} =
  1$, the conclusion follows from Theorem \ref{rearrangethm} since
  $\gkboxurperpdn_\mu = \gkboxrperpdn_\mu$ says $\epsilon = 0$.
\end{proof}

\begin{corollary}[``Bernstein-Szeg\H{o} approximation'']
  \label{BernSzegcor} Let $\rho$ be
  a positive Borel measure satisfying $\gkboxurperpdn_\rho =
  \gkboxrperpdn_\rho$.  Suppose $q \in \gkboxllperp_\rho$ has no
  factors in common with $\refl{q}$ and define
\[
d\mu = \frac{1}{|q(z,w)|^2} d\sigma(z,w).
\]
If we normalize $\rho$ so that $||q||_{L^2(\rho)} = 1$, then
$\gkbox_\rho = \gkbox_\mu$ and
\[
K\gkbox_\rho = K\gkbox_\mu,
\]
i.e. the inner products on $\gkbox_\mu$ and $\gkbox_\rho$ from
$L^2(\mu)$ and $L^2(\rho)$ agree.
\end{corollary}

\begin{proof} By Proposition \ref{qrelations} $q \in
  \gkboxllperp_\mu$ and by Theorem \ref{BernSzegthm}, 
  $\gkboxurperpdn_\mu = \gkboxrperpdn_\mu$.  We have assumed $q$ has
  no factors in common with $\refl{q}$ and this allows us to apply
  Theorem \ref{spectralmatching}, from which the conclusion follows
  immediately.
\end{proof}

One final lemma will make the proof of the main theorem a matter of
bookkeeping.  We use the following notations:
\begin{equation} \label{def:Zq}
Z_q = \{(z,w)\in \mathbb{C}^2: q(z,w)=0\},
\end{equation}
\begin{equation} \label{def:pis}
\pi_1(z,w) = z \text{ and } \pi_2(z,w) = w.
\end{equation}

\begin{lemma} \label{BernSzegdivisors} 
If $\mu$ is the Bernstein-Szeg\H{o} measure associated
  to $q\in \mathbb{C}[z,w]$:
\[
d\mu = \frac{1}{|q(z,w)|^2} d\sigma(z,w)
\]
then $J(z,w)= (w-w_0)$ and $L(z,w) = (z-z_0)$ will be divisors of the
ideal $\mathcal{I}_\mu$ whenever $w_0 \notin
\pi_2(Z_q\cap\mathbb{T}^2)$ and $z_0 \notin
\pi_1(Z_q\cap\mathbb{T}^2)$ respectively.
\end{lemma}

\begin{proof} If $(z-z_0) f(z,w) \in L^2(\mu)$ for some $f \in
  \mathbb{C}[z,w]$ and $z_0 \notin \pi_1(Z_q\cap\mathbb{T}^2)$, then
let $U$ be a neighborhood of $Z_{z-z_0}\cap \mathbb{T}^2$ which does
not intersect $Z_q$.  Then, $|z-z_0|^2$ is bounded below on
$\mathbb{T}^2 \setminus U$ and $|q|^2$ is bounded below on $U$, say by
a constant $c$.  Then,
\[
\infty > \int_{\mathbb{T}^2} \frac{|z-z_0|^2|f(z,w)|^2}{|q(z,w)|^2}
d\sigma \geq \int_{\mathbb{T}^2\setminus U} \frac{c
|f(z,w)|^2}{|q(z,w)|^2} d\sigma
\]
and 
\[
\infty > \int_{U} |f(z,w)|^2 d\sigma \geq \int_{U} \frac{c
  |f(z,w)|^2}{|q(z,w)|^2} d\sigma
\]
together imply
\[
||f||^2_{L^2(\mu)} = \int_{U} \frac{ |f(z,w)|^2}{|q(z,w)|^2} d\sigma +
\int_{\mathbb{T}^2\setminus U} \frac{ |f(z,w)|^2}{|q(z,w)|^2} d\sigma
< \infty.
\]
This proves $L$ is a divisor of $\mathcal{I}_\mu$.  The proof for $J$
is similar.
\end{proof}

\section{Proof of the main theorem} \label{sec:mainthm}

We have all of the pieces in place to prove the theorem from the
introduction.  Here is the main theorem with extra details filled in.
When we use the inner product notation $\ip{\cdot}{\cdot}$ below with
no subscript, we are taking inner products in $\mathbb{C}^N$ (where
the $N$ is taken from context) and not taking any kind of Hilbert
function space inner product.

\begin{theorem} \label{mainthmfull} Let $q\in \mathbb{C}[z,w]$ have degree
  at most $(n,m)$ with no zeros on $\mathbb{D}^2$ and finitely many
  zeros on $\mathbb{T}^2$.  Then, there exist vector polynomials
  $\mathbf{E} \in \mathbb{C}^n[z,w]$ and $\mathbf{F} \in
  \mathbb{C}^m[z,w]$ of degree at most $(n-1,m)$ and $(n,m-1)$
  respectively (in each component) with the property that if we write
  them in matrix form as
\[
\mathbf{E}(z,w) = E(w) \mathbf{\Lambda}_n(z) 
\qquad
\mathbf{F}(z,w) = F(z) \mathbf{\Lambda}_m(w) 
\]
where $E(w)$ is an $n\times n$ matrix polynomial of degree at most $m$
and $F(z)$ is an $m\times m$ matrix polynomial of degree at most $n$,
then
\begin{enumerate}
\item $E(w)$ is invertible for all $w \in \overline{\mathbb{D}}$ with
  the possible exception of $w \in \mathbb{T}\cap\pi_2 (Z_q)$,
\item $z^n \overline{F(1/\bar{z})}$ is invertible for all $z \in
  \overline{\mathbb{D}}$ with the possible exception of $z \in
  \mathbb{T}\cap \pi_1 (Z_q)$,
\item the following formula holds
\begin{align} \label{mainsosfull}
  q(z,w) &\overline{q(Z,W)} - \refl{q}(z,w) \overline{\refl{q}(Z,W)}
  \\ \nonumber
&= (1-z\bar{Z}) \ip{\mathbf{E}(z,w)}{\mathbf{E}(Z,W)} + (1-w\bar{W})
\ip{\mathbf{F}(z,w)}{\mathbf{F}(Z,W)}, 
\end{align}

\item if $\tilde{\mathbf{E}} \in \mathbb{C}^n[z,w]$ and
  $\tilde{\mathbf{F}} \in \mathbb{C}^m[z,w]$ satisfy items (1) and (3)
  above in place of $\mathbf{E}$ and $\mathbf{F}$, then there exist
  unitary matrices $U_1$, $U_2$ such that
\[
\mathbf{E}(z,w) = U_1 \tilde{\mathbf{E}}(z,w) \qquad \mathbf{F}(z,w) =
U_2 \tilde{\mathbf{F}}(z,w),
\]
and
\item there exists $N \leq nm$ and $\mathbf{G} \in
  \mathbb{C}^{N}[z,w]$ such that
\[
\begin{aligned}
&\ip{\mathbf{G}(z,w)}{\mathbf{G}(Z,W)}  \\
&= \frac{\ip{\mathbf{E}(z,w)}{\mathbf{E}(Z,W)} -
  \ip{\mathbf{\refl{E}}(z,w)}{\mathbf{\refl{E}}(Z,W)} }{1-w\bar{W}}\\
& =
\frac{\ip{\mathbf{\refl{F}}(z,w)}{\mathbf{\refl{F}}(Z,W)} -
  \ip{\mathbf{F}(z,w)}{\mathbf{F}(Z,W)} }{1-z\bar{Z}}
\end{aligned}
\]
where 
\[
\mathbf{\refl{E}}(z,w) := z^{n-1} w^m
\overline{\mathbf{E}(1/\bar{z}, 1/\bar{w})} \text{ and }
\mathbf{\refl{F}}(z,w) := z^{n} w^{m-1}
\overline{\mathbf{F}(1/\bar{z}, 1/\bar{w})}.
\]

\end{enumerate}

\end{theorem}

\begin{proof}
  We use the setup (and conclusion) of Corollary \ref{bernszegsos}.
  By Lemma \ref{dimensionlemma}, $\gkboxrperpup_\mu$ has dimension $n$
  and $\gkboxuperplt_\mu$ has dimension $m$.  Let $\{e_1, \dots,
  e_n\}$ be an orthonormal basis of $\gkboxrperpup_\mu$ and $\{f_1,
  \dots, f_m\}$ an orthonormal basis of $\gkboxuperplt_\mu$.  Define
  $\mathbf{E} = (e_1, \dots, e_n)^t \in \mathbb{C}^n[z,w]$ and
  $\mathbf{F} = (f_1, \dots, f_m)^t \in \mathbb{C}^m[z,w]$. Corollary
  \ref{bernszegsos} now proves item (3).

 Write $\mathbf{E}(z,w) = E(w) \bLam_n(z)$ and $\mathbf{F}(z,w) = F(z)
  \bLam_m(w)$.  With these choices, Proposition \ref{fullrankprop}
  says $E(w)$ is invertible for all $w \in \overline{\mathbb{D}}$ with
  the exception of $w_0\in \mathbb{T}$ with the property that $w-w_0$
  is not a divisor of $\mathcal{I}_\mu$.  Lemma \ref{BernSzegdivisors}
  says $(w-w_0)$ is a divisor of $\mathcal{I}_\mu$ when $w_0 \notin
  \pi_2(Z_q\cap \mathbb{T}^2)$. So, $E(w)$ is invertible when $w \in
  \overline{\mathbb{D}}\setminus \pi_2(Z_q \cap \mathbb{T}^2)$.  The
  entries of
\[
\refl{\mathbf{F}}(z,w) = z^nw^{m-1}\overline{\mathbf{F}(1/\bar{z},
    1/\bar{w})} 
\]
form an orthonormal basis for $\gkboxuperprt_\mu$ and 
\[
\refl{\mathbf{F}}(z,w) = z^n\overline{F(1/\bar{z})} w^{m-1}
  \overline{\bLam_m(1/\bar{w})}  = z^n\overline{F(1/\bar{z})} X \bLam_m(w)
\]
where $X$ is the $m\times m$ matrix with ones on the anti-diagonal
(entries $(j,m-j)$) and zeros elsewhere.  By Proposition
\ref{fullrankprop} and Lemma \ref{BernSzegdivisors}
$z^n\overline{F(1/\bar{z})}X$ is invertible for $z \in
\overline{\mathbb{D}}\setminus \pi_1(Z_q \cap \mathbb{T}^2)$. Of
course, $X$ is invertible, so the same statement holds for
$z^n\overline{F(1/\bar{z})}$.  This proves items (1) and (2) of
Theorem \ref{mainthmfull}. 

Lemma
  \ref{uniquenesslemma} proves item (4).  Item (5) follows from the
  fact that
\[
K \gkboxsm_\mu = \frac{K \gkboxuperprt_\mu -
  K\gkboxuperplt_\mu}{1-z\bar{Z}} = \frac{K \gkboxrperpup_\mu - K
  \gkboxrperpdn_\mu}{ 1-w\bar{W}}
\]
and we can factor 
\[
K\gkboxsm_\mu ((z,w),(Z,W)) = \ip{\mathbf{G}(z,w)}{\mathbf{G}(Z,W)}
\]
using an orthonormal basis of $\gkboxsm_\mu$ (a subspace with
dimension at most $nm$).
\end{proof}

\section{Polynomials with unique decompositions}
\label{sec:uniquedecomp}

In this section we give a characterization of the polynomials with no
zeros on the bidisk that have a unique sums of squares
decomposition. 

\begin{proof}[Proof of Theorem \ref{uniquedecompthm}]
Suppose $q$ is a polynomial of degree $(n,m)$ with no zeros on
$\mathbb{D}^2$ and finitely many zeros on $\mathbb{T}^2$.  

To prove item (1) implies (2) in the theorem, suppose there are unique
$\Gamma_1$ and $\Gamma_2$, sums of squared moduli of two variable
polynomials, such that
\[
|q(z,w)|^2 - |\refl{q}(z,w)|^2 = (1-|z|^2)\Gamma_1(z,w) +
 (1-|w|^2)\Gamma_2(z,w).
\]
By Corollary \ref{bernszegsos}, if $\mu$ is the
Bernstein-Szeg\H{o} measure associated to $q$ then
\[
\begin{aligned}
  |q(z,w)|^2& - |\refl{q}(z,w)|^2
  \\
  =& (1-|z|^2) K\gkboxrperpup_\mu ((z,w),(z,w))
  + (1-|w|^2) K\gkboxuperplt_\mu ((z,w),(z,w)) \\
  =& (1-|z|^2) K\gkboxrperpdn_\mu ((z,w),(z,w))
  + (1-|w|^2) K\gkboxuperprt_\mu ((z,w),(z,w)).
\end{aligned}
\]
These reproducing kernels can be written as sums of squares of two
variable polynomials.  Since we are assuming such decompositions are
unique we have
\[
K\gkboxrperpup_\mu ((z,w),(z,w)) = K\gkboxrperpdn_\mu ((z,w),(z,w)).
\]
Because of the formula
\begin{equation} \label{equdthm}
K\gkboxrperpup_\mu ((z,w),(z,w)) - K\gkboxrperpdn_\mu ((z,w),(z,w))
=(1-|w|^2)K\gkboxsm_\mu((z,w),(z,w))
\end{equation}
we see that
\[
K\gkboxsm_\mu((z,w),(z,w)).
\]
This implies $\gkboxsm_\mu = \{0\}$.  In other words, there are no
nonzero $f\in \gkboxsm\cap L^2(\mu) = \gkboxsm \cap L^2(1/|q|^2
d\sigma)$ and this just says there are no nonzero $f \in \gkboxsm$
such that
\[
f/q \in L^2(\mathbb{T}^2).
\]
This proves that item (1) implies item (2) in Theorem
\ref{uniquedecompthm}.  

To prove item (2) implies (3) in the theorem, assume there are no
nonzero $f \in \gkboxsm$ such that
\[
f/q \in L^2(\mathbb{T}^2).
\]
This just says $\gkboxsm_\mu = \{0\}$ and again by \eqref{equdthm} we
have
\[
K\gkboxrperpup_\mu ((z,w),(z,w)) = K\gkboxrperpdn_\mu ((z,w),(z,w)).
\]
The two subspaces $\gkboxrperpup_\mu$ and $\gkboxrperpdn_\mu$ are
reflections of one another.  So, if we write 
\[
K\gkboxrperpup_\mu((z,w),(z,w)) = K\gkboxrperpdn_\mu((z,w),(z,w)) =
|\mathbf{E}(z,w)|^2
\]
where $\mathbf{E}(z,w) = (E_1(z,w),\dots, E_n(z,w))^t \in
\mathbb{C}^n[z,w]$ and $E_1,\dots, E_n$ are an orthonormal basis for
$\gkboxrperpup_\mu=\gkboxrperpdn_\mu$, then the entries of
\[
\refl{\mathbf{E}}(z,w) := z^{n-1}w^m
\overline{\mathbf{E}(1/\bar{z}, 1/\bar{w})}
\]
also form an orthonormal basis for
$\gkboxrperpup_\mu=\gkboxrperpdn_\mu$.  This implies
\[
|\mathbf{E}(z,w)|^2 = |\mathbf{\refl{E}}(z,w)|^2
\]
and by Lemma \ref{simpleunique} there is an $n\times n$ unitary matrix
$U$ such that
\[
U\mathbf{E}(z,w) = \refl{\mathbf{E}}(z,w).
\]
(As we commented there Lemma \ref{simpleunique} holds for two variable
polynomials just as well.)  If we reflect both sides of this equation
(take conjugates, replace $(z,w)$ with $(1/\bar{z},1/\bar{w})$, and
multiply through by $z^{n-1}w^m$) we see that
\[
\bar{U} \refl{\mathbf{E}}(z,w) = \mathbf{E}(z,w).
\]
Note that $\bar{U}$ is the matrix obtained by taking complex
conjugates of each entry of $U$ and is not the adjoint of $U$.  In
fact, $\bar{U}^{-1} = U^t$ and therefore
\[
U^t \mathbf{E}(z,w) = \refl{\mathbf{E}}(z,w) = U \mathbf{E}(z,w).
\]
This implies $U=U^t$ since the vectors $\mathbf{E}(z,w)$ span all of
$\mathbb{C}^n$ as $(z,w)$ varies over $\mathbb{C}^2$ (by Lemma
\ref{fullrankprop}).  This says $U$ is a symmetric unitary.  Symmetric
unitaries can be factored as $U=V^t V$ where $V$ is a unitary---this
is the so-called Takagi factorization.  The
vector polynomial
\[
V\mathbf{E}(z,w)
\]
is then symmetric since its reflection is
\[
\bar{V}\refl{\mathbf{E}}(z,w) = (V^t)^{-1} U \mathbf{E}(z,w) = V
\mathbf{E}(z,w)
\]
as $U = V^t V$.  So we replace $\mathbf{E}$ with $V\mathbf{E}$ and
this proves there exists a symmetric vector polynomial $\mathbf{E}$
such that
\[
K\gkboxrperpup_\mu((z,w),(z,w)) = K\gkboxrperpdn_\mu((z,w),(z,w))
=|\mathbf{E}(z,w)|^2.
\]
By Proposition \ref{fullrankprop}, if we write $\mathbf{E}(z,w) = E(w)
\bLam_n(z)$, then $E(w)$ is invertible on the disk $\mathbb{D}$ and
on $\mathbb{C}\setminus \overline{\mathbb{D}}$; i.e. 
\[
\det E(w)
\]
has all of its roots on the unit circle $\mathbb{T}$.  

Similar arguments show that when $\gkboxsm_\mu = \{0\}$, there exists
a symmetric vector polynomial $\mathbf{F} \in \mathbb{C}^m[z,w]$ of
degree $(n,m-1)$ with the property that when we write $\mathbf{F}$ as
$F(z)\bLam_m (w)$,
\[
\det F(z)
\]
has all of its roots on the unit circle $\mathbb{T}$ and
\[
K\gkboxuperprt_\mu((z,w),(z,w)) = K\gkboxuperplt_\mu((z,w),(z,w)) =
|\mathbf{F}(z,w)|^2.
\]
By Corollary \ref{bernszegsos}, we have that
\begin{equation} \label{3implies1eq}
|q(z,w)|^2 - |\refl{q}(z,w)|^2 = (1-|z|^2)|\mathbf{E}(z,w)|^2 +
 (1-|w|^2)|\mathbf{F}(z,w)|^2
\end{equation}
where $\mathbf{E}$ and $\mathbf{F}$ satisfy all of the desired
properties.  This proves item (2) implies item (3).

To prove item (3) implies (1) assume \eqref{3implies1eq} holds where
$\mathbf{E}(z,w) = E(w) \bLam_n(z)$, $\mathbf{F}(z,w) = F(z)
\bLam_m(w)$, and both $E(w)$ and $F(z)$ are invertible in the disk.
We must show this is the only sums of squares decomposition for $q$.

Suppose there are vector polynomials $\mathbf{A} \in
\mathbb{C}^N[z,w], \mathbf{B} \in \mathbb{C}^M[z,w]$ such that
\[
|q(z,w)|^2 - |\refl{q}(z,w)|^2 = (1-|z|^2)|\mathbf{A}(z,w)|^2 +
 (1-|w|^2)|\mathbf{B}(z,w)|^2.
\]
Setting $|w|=1$, equation \eqref{3implies1eq} implies
\[
|\mathbf{E}(z,w)|^2 = |\mathbf{A}(z,w)|^2
\]
for $(z,w) \in \mathbb{C}\times \mathbb{T}$.  Since $E(w)$ is
invertible in $\mathbb{D}$, Lemma \ref{hardunique} applies: $n\leq N$
and there exists a one variable $N\times n$ matrix valued valued
rational inner function $\Psi_1$ such that
\[
\mathbf{A}(z,w) = \Psi_1(w) \mathbf{E}(z,w) \text{ for } (z,w) \in
\mathbb{D}^2.
\]
By similar reasoning, $m \leq M$ and there exists an $M\times m$
matrix valued rational inner function $\Psi_2$ such that
\[
\mathbf{B}(z,w) = \Psi_2(z) \mathbf{F}(z,w).
\]
So,
\[
\begin{aligned}
|\mathbf{A}(z,w)|^2 &\leq |\mathbf{E}(z,w)|^2 \\
|\mathbf{B}(z,w)|^2 &\leq |\mathbf{F}(z,w)|^2
\end{aligned}
\]
for all $(z,w) \in \mathbb{D}^2$.  However, we must have equality at
every point in both of these inequalities because otherwise
\[
\begin{aligned}
& (1-|z|^2)|\mathbf{E}(z,w)|^2 + (1-|w|^2)|\mathbf{F}(z,w)|^2 \\
=& (1-|z|^2)|\mathbf{A}(z,w)|^2 + (1-|w|^2)|\mathbf{B}(z,w)|^2
\end{aligned}
\]
would be violated.  Hence, the sums of squares terms for $q$ are
unique:
\[
\begin{aligned}
|\mathbf{A}(z,w)|^2 &= |\mathbf{E}(z,w)|^2 \\
|\mathbf{B}(z,w)|^2 &= |\mathbf{F}(z,w)|^2
\end{aligned}
\]
for all $(z,w) \in \mathbb{C}^2$.  This proves (3) implies (1) and
concludes the proof.
\end{proof}

Corollary \ref{uniquedecompcor} says that the among polynomials with no
zeros on the closed bidisk, the only ones with a unique decomposition
are one variable polynomials.   We prove this now.

\begin{proof}[Proof of Corollary \ref{uniquedecompcor}]  
Suppose $p$ is a polynomial of degree $(n,m)$ with no zeros on the
closed bidisk.  It is implicit in most of this paper that $n,m>0$.  By
Theorem \ref{uniquedecompthm}, since $1/|p|^2$ is integrable, it
follows that $p$ does not have a unique sums of squares decomposition.
If $n=0$ or $m=0$ then $p$ is really just a one variable polynomial
with no zeros on closed disk.  It is well known that the decomposition
in the one variable Christoffel-Darboux formula is unique, since the
sums of squares term can just be solved for; it equals
\[
\frac{|p(z)|^2 - |\refl{p}(z)|^2}{1-|z|^2}
\]
in the case where $m=0$. 
\end{proof}

\section{Fejer-Riesz factorization} \label{sec:Fejer}
In this section we reprove Geronimo and Woerdeman's characterization
of the positive two variable trigonometric polynomials $t$ that have a
Fej\'er-Riesz factorization; i.e. which $t$ can be written as
$t=|p|^2$ where $p$ is a polynomial with no zeros on the closed
bidisk.  Our proof does not make use of a certain ``maximal entropy
result'' and is therefore self-contained.  We also use this as an
opportunity to extend this theorem to the certain cases of
\emph{non-negative} trigonometric polynomials.

We emphasize that the condition $\aocond$ below, can be rephrased as a
relation on the moments of $\mu$ in the case where $\mu$ is a finite
measure.

\begin{theorem}[Geronimo-Woerdeman \cite{GW04}]
Let $t:\mathbb{T}^2 \to \mathbb{C}$ be a positive trigonometric
polynomial of two variables with Fourier coefficients $\hat{t}(j,k)$
supported on the set $|j|\leq n, |k|\leq m$.  Then, there exists $p\in
\mathbb{C}[z,w]$ of degree at most $(n,m)$ with no zeros on the closed
bidisk satisfying $t(z,w) = |p(z,w)|^2$ for all $(z,w)\in
\mathbb{T}^2$ if and only if the measure $d\mu = \frac{1}{t}d\sigma$
satisfies
\[
\gkboxurperpdn_\mu = \gkboxrperpdn_\mu.
\]
\end{theorem}

\begin{proof} The ``only if'' direction follows from Theorem
  \ref{BernSzegthm}.  To prove the ``if'' direction, observe that if
  $\mu$ satisfies $\gkboxurperpdn_\mu = \gkboxrperpdn_\mu$, then by
  Corollary \ref{BernSzegcor}, if $p$ is a unit norm polynomial in
  $\gkboxllperp_\mu$, then $p$ has no zeros on the closed bidisk (see
  Remark \ref{stableremark}) and defining
\[
d\rho = \frac{1}{|p(z,w)|^2} d\sigma
\]
we have that the inner products on $L^2(\mu)$ and $L^2(\rho)$ agree
when restricted to $\gkbox$.  This implies the moments
\[
\int_{\mathbb{T}^2} z^jw^k d\mu = \int_{\mathbb{T}^2} z^j w^k d\rho
\]
for $|j|\leq n, |k|\leq m$.  Here is where we deviate from the
Geronimo-Woerdeman proof.  Observe that
\begin{align}
  1 & = \int_{\mathbb{T}^2} \frac{|p(z,w)|}{\sqrt{t(z,w)}}
  \frac{\sqrt{t(z,w)}}{|p(z,w)|} d\sigma \nonumber \\
  & \leq \sqrt{\int_{\mathbb{T}^2} \frac{|p(z,w)|^2}{t(z,w)} d\sigma}
  \sqrt{ \int_{\mathbb{T}^2} \frac{t(z,w)}{|p(z,w)|^2} d\sigma}
\label{cauchy-schwarz} \\
  & = ||p||_{L^2(\mu)} \sqrt{||t||_{L^1(\rho)}} \nonumber
\end{align}
by Cauchy-Schwarz.  Now, $||p||_{L^2(\mu)} = 1$ since $p$ was chosen
to have unit norm, and since the moments of $\mu$ and $\rho$ agree,
\[
||t||_{L^1(\rho)} = ||t||_{L^1(\mu)} = \int_{\mathbb{T}^2}
  \frac{t(z,w)}{t(z,w)} d\sigma = 1.
\]
Therefore, we have equality in the above application of Cauchy-Schwarz
(equation \eqref{cauchy-schwarz}).  This implies $|p|/\sqrt{t}$ and
$\sqrt{t}/|p|$ are multiples of one another.  This implies $|p|^2 = c
t$ for some constant $c$ and this constant must be $c=1$ since $p$ has
unit norm in $L^2(\mu)$.  Hence, $t(z,w) = |p(z,w)|^2$ for $(z,w) \in
\mathbb{T}^2$.  
\end{proof}

We would like to extend this result to the case of non-negative
trigonometric polynomials, and we have some results in this direction.
Work on characterizing when a non-negative operator-valued two
variable polynomial has a Fej\'er-Riesz type factorization was done in
Dritschel-Woerdeman \cite{DW05}.  (Although the subtleties of all of the
different candidates for the notion of ``outerness'' in several
variables seem to have prevented getting a necessary and sufficient
condition for a Fej\'er-Riesz factorization in that paper.)

We believe that any Fej\'er-Riesz type factorization for non-negative
two variable trigonometric polynomials should take into account the
notions of \emph{toral} and \emph{atoral} polynomials. These notions
were alluded to in Remark \ref{finiteremark}.

\begin{example}  
Consider the non-negative trigonometric polynomial $t(z,w) = |z-w|^2$.
It cannot be factored as $|p(z,w)|^2$ where $p \in \mathbb{C}[z,w]$
has no zeros on the bidisk, because $p$ would necessarily vanish on
the set $\{(z,w) \in \mathbb{T}^2: z=w\}$ and therefore $z-w$ would
divide $p$.  So, the polynomial $zw t(z,w) = 2zw-z^2-w^2$
associated to $t$ has a toral factor, and since this toral factor has
zeros in the bidisk, there is no hope for such a Fej\'er-Riesz type of
factorization.  So, the question of whether a Fej\'er-Riesz
factorization exists depends on the properties of the toral factors of
$t$.  This is true more generally.
\EOEx
\end{example}

Let $t: \mathbb{T}^2 \to \mathbb{C}$ be a non-negative trigonometric
polynomial of two variables:
\[
t(z,w) = \sum_{j=-N}^{N} \sum_{k=-M}^{M} t_{jk} z^j w^k \geq 0
\]
and let $q(z,w) := z^N w^M t(z,w) \in \mathbb{C}[z,w]$.

\begin{lemma} If $q$ has an irreducible toral factor $p$, then $p^2$
  divides $q$, and $t/|p|^2$ is a non-negative trigonometric
  polynomial.
\end{lemma}
\begin{proof} Write $q = hp$ for some $h\in \mathbb{C}[z,w]$.  By
  definition of \emph{toral}, $p$ has infinitely many zeros on
  $\mathbb{T}^2$.  The lemma is not difficult in the case where $p$ is
  a linear polynomial in one variable alone, so we assume this is not
  the case.  Suppose $p$ has degree $(n,m)$.  Let $(z_0,w_0)\in
  \mathbb{T}^2\cap Z_p$ with the property that $p(\cdot,w_0)$ has a
  zero of multiplicity one at $z_0$ and $t(\cdot,w_0)$ is not
  identically zero; this will be the case for all but finitely many of
  the $(z,w) \in \mathbb{T}^2\cap Z_p$.  Now, $t(z,w_0) = z^{-N}
  w_0^{-M} h(z,w_0) p(z,w_0)$, and as $t(\cdot,w_0)$ is a non-negative
  trig polynomial of one variable, it must have zeros of even order on
  $\mathbb{T}$.  Hence, $h(z_0,w_0) = 0$.  Therefore, $h$ and $p$
  share infinitely many zeros, and this implies $p$ divides $h$ by
  irreducibility of $p$.  Hence, $p^2$ divides $q$.  Toral polynomials
  are $\mathbb{T}^2$-symmetric in the sense that
\[
\refl{p} = c p 
\]
for some unimodular constant $c$.  So, $t(z,w) = z^{-N} w^{-M}
p(z,w)^2 g(z,w) = z^{-N+n}w^{-M+m} |p(z,w)|^2 g(z,w)$ for some $g \in
\mathbb{C}[z,w]$.  Thus, $t/|p|^2$ is a non-negative trig polynomial.
\end{proof}

\begin{corollary} If $t$ is a non-negative trigonometric polynomial,
  then $t$ can be factored into $t(z,w) = |p(z,w)|^2 s(z,w)$ where $p
  \in \mathbb{C}[z,w]$ is a toral polynomial (or is a constant) and
  $s$ is a non-negative trigonometric polynomial with finitely many
  zeros on $\mathbb{T}^2$.
\end{corollary}

This corollary divides the study of characterizing trig polynomials
with a Fej\'er-Riesz factorization into the question of when a toral
polynomial has no zeros on the bidisk and when a non-negative trig
polynomial finitely many zeros on the torus has a Fej\'er-Riesz
factorization. 

To introduce the next result we recall that every positive two
variable trigonometric polynomial can be written as a sum of squares
of two variable polynomials. This was proved in Dritschel \cite{mD04}
and reproved in Geronimo-Lai \cite{GL06} (this latter paper has a
summary of related known results).  It is unknown if all non-negative
trigonometric polynomials can be written as a sum of squares of two
variable polynomials.  The above corollary says that it is enough to
address this question for trig polynomials with finitely many zeros.
On the other hand, if it is true that all non-negative trig
polynomials are equal to a sum of squares of polynomials, then our
approach allows us to characterize when they can be written as a
single square of a polynomial with no zeros on the bidisk.

\begin{theorem} Suppose $p_1, \dots, p_N \in \mathbb{C}[z,w]$ have
  degree at most $(n,m)$ and no common factor.  Also, assume that for
  some $j$, $p_j(0,0) \ne 0$.  Let
\[
t(z,w) = \sum_{j=1}^{N} |p_j(z,w)|^2 \text{ for } (z,w) \in \mathbb{T}^2
\]
and define $d\mu = \frac{1}{t} d\sigma$.  The trigonometric polynomial $t$ can
be written as $t(z,w)= |p(z,w)|^2$, where $p$ has no zeros on the
bidisk, if and only if 
\[
\aocond
\]
\end{theorem}

If every $p_j$ vanishes at the origin, we could apply a M\"obius
transformation to make sure not all of the polynomials vanish at the
origin and then apply the above theorem to check whether the trig
polynomial has the desired factorization.  

\begin{proof} 
Our proof in the case of a strictly positive trig polynomial carries
over with some modifications.  The ``only if'' direction again follows
from Theorem \ref{BernSzegthm}.  Let us prove that $\aocond$ implies
$t$ has a Fej\'er-Riesz type of decomposition.  

Since $t$ is of the given form it is clear that each $p_j \in
L^2(\mu)$, as $|p_j|^2/t \leq 1$ on the torus. The assumption that
$p_j(0,0) \ne 0$ guarantees that $\gkboxllperp_\mu$ is nonempty (since
we then know $\gkboxll_\mu \ne \gkbox_\mu$).  Let $q$ be a unit norm
polynomial in $\gkboxllperp_\mu$.  By Corollary \ref{factorcorollary},
$q$ has no zeros on the bidisk and finitely many zeros on the
torus. (The corollary says $q$ can be factored as $q_1q_2$ where $q_1$
divides every element of $\gkbox_\mu$ and $q_2$ is of the desired
type, but we assumed $p_1,\dots, p_N$ have no common factor.  Hence,
$q_1$ must be a constant.) Define
\[
d\rho = \frac{1}{|q(z,w)|^2} d\sigma.
\]
By Corollary \ref{BernSzegcor}, $\gkbox_\mu = \gkbox_\rho$ and the
inner products of $L^2(\mu)$ and $L^2(\rho)$ agree on $\gkbox_\mu$.
This says in particular that
\[
p_j/q \in L^2(\mathbb{T}^2)
\]
for each $j$.  Just as in the proof in the strictly positive case, we
can prove
\[
1 \leq ||q||_{L^2(\mu)} \sqrt{||t||_{L^1(\rho)}}
\]
by an application of Cauchy-Schwarz.  Since $q$ has unit norm,
$||q||_{L^2(\mu)} =1$, and since the inner products agree, we have
\[
||t||_{L^1(\rho)} = \sum_{j=1}^{N} ||p_j||^2_{L^2(\rho)} =
  \sum_{j=1}^N ||p_j||^2_{L^2(\mu)} = ||t||_{L^1(\mu)} = 1.
\]
Therefore, just as in the proof for the strictly positive case, we
have equality in Cauchy-Schwarz, which implies $t = |q|^2$ on the
torus.  
\end{proof}

So, the above theorem addresses non-negative trig polynomials of a
specific form.  The above proof would also work if we could decompose
$t$ as
\[
t(z,w) = \sum_{j=1}^{N} p_j(z,w)\overline{q_j(z,w)}
\]
where $p_j, q_j \in L^2(\frac{1}{t} d\sigma)$ have no common factor and not
all vanish at $(0,0)$.  

\begin{question} \label{trigquestion} Can \emph{every} non-negative two
  variable trigonometric polynomial $t$ be decomposed as
\[
 t(z,w) = \sum_{j=1}^{N} p_j(z,w)\overline{q_j(z,w)}
\]
where $p_j, q_j$ are in $L^2(\frac{1}{t} d\sigma)$ and have no common factor?
\end{question}

Next, we tackle toral factors of non-negative trig polynomials.

\begin{theorem} An irreducible toral polynomial $p \in
  \mathbb{C}[z,w]$ has no zeros in the bidisk if and only if 
\[
\refl{\frac{\partial p}{\partial z}} + \refl{\frac{\partial
    p}{\partial w}}
\]
has no zeros in the closed bidisk and finitely many zeros on the
torus.  In this case, all of the zeros occur at singularities of $Z_p$
(i.e. common zeros of $\frac{\partial p}{\partial z}$ and
$\frac{\partial p}{\partial w}$).
\end{theorem}

The above reflections are performed at the degrees of $\partial
p/\partial z$ and $\partial p/\partial w$ that would generically be
expected.  Namely, if $p$ has degree $(n,m)$, we reflect $\partial
p/\partial z$ at the degree $(n-1,m)$.

\begin{proof} If $p$ is toral, then $p$ is necessarily $\mathbb{T}^2$
  symmetric, meaning $p$ is a unimodular constant times $\refl{p}$
  (and in fact we may assume $p = \refl{p}$ by multiplying by an
  appropriate constant).  It is proved in Knese \cite{gK08b} that if
  $p$ is $\mathbb{T}^2$ symmetric and has no zeros in the bidisk, then
\[
\refl{\frac{\partial p}{\partial z}} + \refl{\frac{\partial
    p}{\partial w}}
\]
has no zeros in the set $\cbidisk$ except possibly at singularities of
$Z_p$ (and there can be at most finitely many singularities).  

Conversely, suppose $\refl{\frac{\partial p}{\partial z}} +
\refl{\frac{\partial p}{\partial w}}$ has no zeros in the bidisk and
finitely many zeros on the torus.  This implies
\[
\phi(z,w) = \frac{z\frac{\partial p}{\partial z}(z,w) + w \frac{\partial
  p}{\partial w}(z,w)}{\refl{\frac{\partial p}{\partial z}}(z,w) + \refl{\frac{\partial
    p}{\partial w}}(z,w)}
\]
is a (non-constant) inner function on the bidisk, and must be bounded
by 1 in modulus on the bidisk.

It is also proved in Knese \cite{gK08b} that if $p$ is $\mathbb{T}^2$
symmetric, then
\[
(n+m)p(z,w) = z\frac{\partial p}{\partial z}(z,w) + w \frac{\partial
  p}{\partial w}(z,w) +\refl{\frac{\partial p}{\partial z}}(z,w) + \refl{\frac{\partial
    p}{\partial w}}(z,w).
\]
So, if $p(z,w)=0$ for some $(z,w) \in \mathbb{D}^2$, then $|\phi(z,w)|
= 1$, which is a contradiction.  Therefore, $p$ has no zeros in the bidisk.
\end{proof}

\begin{remark} We view this as progress on determining which
  non-negative trig polynomials have a Fej\'er-Riesz decomposition for
  the following reasons. A non-negative trig polynomial has a unique
  toral factor $|p|^2$ and determining whether $p$ has no zeros in the
  bidisk can be approached by looking at each factor of $p$. For the
  factors $f$ whose zero sets have no singularities on the torus, the
  above theorem says we can check whether $\refl{\frac{\partial
  f}{\partial z}} + \refl{\frac{\partial f}{\partial w}}$ has no zeros
  on the closed bidisk.  This can be accomplished by using a two
  variable Schur-Cohn test, such as the one presented in
  Geronimo-Woerdeman \cite{GW04}. For factors with singularities on
  the torus, one would need to adapt the Schur-Cohn test to test for
  no zeros on the closed bidisk with the exception of finitely many
  zeros on the torus.  We leave this for future work.
\end{remark} 

To summarize, given a non-negative trig polynomial $t$ we can factor
it into $t(z,w) = |p(z,w)|^2 s(z,w)$ where $p$ is a toral polynomial
and $s$ is a non-negative trig polynomial with finitely many zeros on
$\mathbb{T}^2$.  The above remark addresses cases where we can
determine whether $p$ has no zeros in the bidisk.  If $s$ has no zeros
on the torus, the Geronimo-Woerdeman theorem characterizes whether it
can be factored as $|q|^2$ where $q$ has no zeros on the closed
bidisk.  We have extended this characterization to a class of
non-negative trig polynomials with a special form, for which it is
unknown whether this is all non-negative trig polynomials.

\section{Application to Distinguished Varieties} \label{distvar}

One of our main applications is a bounded analytic extension theorem
for \emph{distinguished varieties}, which we now define.

\begin{definition} \label{def:distvar} 
A nonempty subset $V\subset \mathbb{C}^2$ is a \emph{distinguished
  variety} if $V$ is an algebraic curve: there exists $p \in
  \mathbb{C}[z,w]$ such that
\[
V = \{(z,w) \in \mathbb{C}^2: p(z,w) = 0\}
\]
and $V$ exits the bidisk through the distinguished boundary:
\[
\partial (V \cap \overline{\mathbb{D}^2}) \subset \mathbb{T}^2.
\]
\end{definition}

It is proved in Knese \cite{gK08b} that if $V$ is defined via a
polynomial $p$ of minimal degree then 
\[
V \subset \mathbb{D}^2 \cup \mathbb{T}^2 \cup \mathbb{E}^2
\]
where $\mathbb{E} = \mathbb{C}\setminus \overline{\mathbb{D}}$.

In Knese \cite{gK08b}, we proved that if $V$ is a distinguished
variety with no singularities on $\mathbb{T}^2$, then every polynomial
$f \in \mathbb{C}[z,w]$, considered as a function on $V\cap
\mathbb{D}^2$, has an extension to a rational function $F$ on
$\mathbb{D}^2$ such that
\[
\sup_{\mathbb{D}^2} |F| \leq C \sup_{V\cap \mathbb{D}^2} |f|
\]
for some constant $C$.  We extend this result to \emph{all}
distinguished varieties (i.e. singularities are allowed) in Theorem
\ref{extendthm} below.  The price we pay is that instead of getting a
constant increase in norm, we control the growth of the extended
function.  Before we present the theorem a little background is
required.

The use of the Cole-Wermer sums of squares formula is essential to the
work in Knese \cite{gK08b}, and if we use Theorem \ref{mainthm} in its
place, the following lengthy theorem can be proved by slightly
modifying the proofs in \cite{gK08b}.

\begin{theorem} \label{sosdist} Let $V$ be a distinguished variety
  given as the zero set of a polynomial $p \in \mathbb{C}[z,w]$ of
  degree $(n,m)$. Let $a,b > 0$ be positive real numbers.  Then,

\begin{itemize}
\item there exists a vector polynomial $\vec{P} \in
  \mathbb{C}^{n}[z,w]$ of degree at most $(n-1,m)$ and a vector
  polynomial $\vec{Q} \in \mathbb{C}^m[z,w]$ of degree at most
  $(n,m-1)$ such that
\[
\begin{aligned}
  (bm-an) |p(z,w)|^2 &+ 2\text{Re} [(a z\frac{\partial p}{\partial z}(z, w)-b w
    \frac{\partial p}{\partial w}(z,w))\overline{p(z,w)}]\\ & +
  (1-|z|^2) |\vec{P}(z,w)|^2 \\ =&
  (1-|w|^2) |\vec{Q}(z,w)|^2,
\end{aligned}
\]
\item if $p$ is a product of distinct irreducible factors, then none of the
entries of $\vec{P}$ or $\vec{Q}$ can vanish identically on $V$,

\item  there is a $m\times m$ matrix-valued rational inner function
  $\Phi: \mathbb{D} \to \mathbb{C}^{m\times m}$ such that $V$ has the
  following representation
\[
V\cap \mathbb{D}^2 = \{(z,w) \in \mathbb{D}^2: \det(wI_m - \Phi(z)) =
0\}
\]

\item $\vec{Q}$ can be chosen to have at most finitely many zeros on $V$
  and to satisfy
\[
\Phi(z) \vec{Q}(z,w) = w \vec{Q}(z,w)
\]
for all $(z,w) \in V$ and when we write
\[
\vec{Q}(z,w) = Q(z) \mathbf{\Lambda}_m(w)
\]
where $Q(z)$ is an $m\times m$ matrix polynomial of degree at most $n$
in each entry, we have that $Q(z)$ is invertible for all $z \in
\mathbb{D}$ and for all $z \in \mathbb{T}$ with the exception of $z
\in \pi_1(S)$, where $S$ is the set of singularities of $V$. In
particular, $\vec{Q}(z,w)$ has no zeros in $\mathbb{D}^2$.
\end{itemize}

\end{theorem}

\begin{proof}[Guide to the proof]  Everything above is contained in a
  theorem in Knese \cite{gK08b} except for the condition that $Q(z)$
  is invertible for all $z\in \mathbb{D}$, so let us briefly outline
  how all of this can be done.  All of the following are proved in
  Knese \cite{gK08b}:
\begin{enumerate}
\item If $p\in \mathbb{C}[z,w]$ has degree $(n,m)$ and defines a
distinguished variety, then the polynomial
\[
q(z,w) = z^np(\frac{1}{z},w)
\]
is $\mathbb{T}^2$-symmetric and has no zeros on the bidisk.

\item Such a $q$ has the property that for each $a,b>0$
\[
a\refl{\frac{\partial q}{\partial z}} + b \refl{\frac{\partial
    q}{\partial w}}
\]
has no zeros on the closed bidisk $\cbidisk$ except possibly at
the finite number of singularities of $Z_q$, which necessarily occur
on $\mathbb{T}^2$.

\item Such a $q$ satisfies

\begin{align}
&(an+bm)^2|q(z,w)|^2 - 2\text{Re}[(azq_z(z,w)+bwq_w(z,w)) (an+bm)
\overline{q(z,w)}] \nonumber \\
&= |a\refl{\frac{\partial q}{\partial z}}(z,w) + b\refl{\frac{\partial
q}{\partial w}}(z,w)|^2 - |az\frac{\partial q}{\partial z}(z,w) + b w
  \frac{\partial q}{\partial w} (z,w)|^2. \label{lastitem}
\end{align}

\end{enumerate}

By Theorem \ref{mainthmfull}, this last item \eqref{lastitem} can
written as
\[
(1-|z|^2)|\mathbf{E}(z,w)|^2+(1-|w|^2)|\mathbf{F}(z,w)|^2
\]
where $\mathbf{E}$ and $\mathbf{F}$ satisfy the conditions in Theorem
\ref{mainthmfull} (actually we need $w^m\overline{E(1/\bar{w})}$ to be
invertible and $\mathbb{D}$ and $F(z)$ invertible in $\mathbb{D}$, but
this can be arranged).  If we convert back to statements involving the
polynomial $p$ (by replacing $z$ with $1/z$ and multiplying by $z^n$)
we get
\[
\begin{aligned}
  (bm-an) |p(z,w)|^2 &+ 2\text{Re} [(a z\frac{\partial p}{\partial z}(z, w)-b w
    \frac{\partial p}{\partial w}(z,w))\overline{p(z,w)}]\\ & +
  (1-|z|^2) |\vec{P}(z,w)|^2 \\ =&
  (1-|w|^2) |\vec{Q}(z,w)|^2,
\end{aligned}
\]
 where if we write $\vec{Q}(z,w) = Q(z)\mathbf{\Lambda}_m(w)$, we have that
 $Q(z)$ is invertible in $\overline{\mathbb{D}}$ except at first
 coordinates of singular points of $V$ on $\mathbb{T}^2$.  For the rest of
 the theorem, the proofs in Knese \cite{gK08b} can be applied
 unchanged. 
\end{proof}

Here is the promised ``bounded analytic extension'' theorem. The proof
is identical to the proof in Knese \cite{gK08b} for distinguished
varieties with no singularities on the torus.  The only difference is
that in that case $Q(z)$ is invertible on the closed disk and
therefore the quantity 
\[
\sup_{\mathbb{D}} ||Q(z)^{-1}||\  ||Q(z)||
\]
was finite.  

\begin{theorem} \label{extendthm} Let $V$ be a distinguished variety
  and let $\Phi$, $Q$, and $\vec{Q}$ be as in Theorem \ref{sosdist}.
  Then, for any polynomial $f \in \mathbb{C}[z,w]$, the rational
  function
\[
F(z,w) := (1,0, \dots, 0) Q(z)^{-1} f(zI_m, \Phi(z)) \vec{Q}(z,w)
\]
is equal to $f$ on $V\cap\mathbb{D}^2$ and we have the estimates
\[
\begin{aligned}
|F(z,w)| &\leq ||Q(z)^{-1}||\ |\vec{Q}(z,w)| \sup_{V\cap\mathbb{D}^2} |f|\\
& \leq \sqrt{m} ||Q(z)^{-1}||\  ||Q(z)|| \sup_{V\cap\mathbb{D}^2} |f| \\
\end{aligned}
\]
for all $(z,w) \in \mathbb{D}^2$.  Here we are taking the operator
norm of the matrices $Q(z)$ and $Q(z)^{-1}$.
\end{theorem}

In words, the growth of the extension $F$ is controlled by a rational
function of one variable.  We believe there is some novelty to this
theorem; it seems ``extension theorems'' for holomorphic functions on
varieties vary between the very general but non-explicit sheaf
cohomological methods (e.g. see Corollary 10.5.4 in Taylor
\cite{jT02}) and explicit integral formula approaches which require no
singularities on the boundary in order to be able to make estimates
(see e.g. Adachi-Andersson-Cho \cite{AAC99}).  Our theorem is
essentially algebraic and applies without any assumptions about
singularities on the boundary.

\section{Necessity in Agler's Pick interpolation
  theorem} \label{sec:agler} As another application we give a simple
proof of necessity in the Pick interpolation theorem on the
bidisk. This proof sidesteps the use of And\^{o}'s inequality and
cone-separation arguments found in most proofs.  (The proof of
sufficiency can be accomplished with a ``lurking isometry'' argument;
see Lemma \ref{dimensionlemma} for something similar.)  The proof is
very similar to the argument in Cole-Wermer \cite{CW99} for
establishing And\^{o}'s inequality from the sum of squares
decomposition.

\begin{theorem}[Agler] \label{aglerpick} Given distinct points
  $(z_1,w_1), \dots, (z_N,w_N) \in \mathbb{D}^2$ and complex numbers
  $c_1, \dots, c_N \in \mathbb{D}$, there exists a holomorphic
  function $f:\mathbb{D}^2 \to \mathbb{D}$ which interpolates
\[
f(z_j,w_j) = c_j \text{ for } j=1,2,\dots, N
\]
if and only if there exist positive semi-definite $N\times N$ matrices
$\Gamma$ and $\Delta$ such that
\[
1-c_j \bar{c_k} = (1-z_j \bar{z_k})\Gamma_{jk} + (1-w_j\bar{w_k})
\Delta_{jk}
\]
\end{theorem}

\begin{proof}[Proof of necessity:] We first prove the theorem for
  rational inner functions and then use an approximation theorem to
  prove necessity in general.  So, let $f$ be a rational inner
  function on the bidisk.  Every rational inner function can be
  written as $f = \refl{p}/p$ for some $p \in \mathbb{C}[z,w]$ of
  degree at most $(n,m)$ having no zeros on the bidisk (see Rudin
  \cite{wR69} Theorem 5.5.1).  Decomposing $p$ as in \eqref{mainsosfull}
  and setting $(z,w)=(z_j,w_j)$ and $(Z,W) = (z_k,w_k)$ we have
\[
\begin{aligned}
  p(z_j,w_j)&\overline{p(z_k,w_k)} -
  \refl{p}(z_j,w_j)\overline{\refl{p}(z_k,w_k)} \\
  =& (1-z_j\bar{z_k}) \ip{\mathbf{E}(z_j,w_j)}{\mathbf{E}(z_k,w_k)} \\
  &+ (1-w_j\bar{w_k}) \ip{\mathbf{F}(z_j,w_j)}{\mathbf{F}(z_k,w_k)}.
\end{aligned}
\]
Therefore, if $f(z_j,w_j) = (\refl{p}/p)(z_j,w_j) = c_j$, then 
\[
\Gamma_{jk} = \frac{1}{p(z_j,w_j)\overline{p(z_k,w_k)}}
\ip{\mathbf{E}(z_j,w_j)}{\mathbf{E}(z_k,w_k)}
\]
and
\[
\Delta_{jk} = \frac{1}{p(z_j,w_j)\overline{p(z_k,w_k)}}
\ip{\mathbf{F}(z_j,w_j)}{\mathbf{F}(z_k,w_k)}
\]
are both positive semi-definite matrices and they satisfy
\begin{equation} \label{aglerdecomp}
1-c_j \bar{c_k} = (1-z_j \bar{z_k})\Gamma_{jk} + (1-w_j\bar{w_k})
\Delta_{jk}
\end{equation}
as desired.

In general, suppose $f:\mathbb{D}^2 \to \mathbb{D}$ is holomorphic and
$f(z_j,w_j) = c_j$.  Rudin's extension of Carath\'{e}odory's theorem
to the polydisk (see Theorem 5.5.1 of Rudin \cite{wR69} ), says that
$f$ is the pointwise limit of a sequence of rational inner functions:
$f_\alpha \to f$ as $\alpha \to \infty$, where $\alpha$ is used to
index the positive integers.  Corresponding to each such rational
inner function $f_\alpha$, we write $f_\alpha(z_j,w_j) =
c_{\alpha,j}$ and we choose positive semi-definite matrices
$\Gamma_{\alpha}, \Delta_{\alpha}$ so that an equation analogous to
\eqref{aglerdecomp} holds:
\begin{equation} \label{alphadecomp} 1-c_{\alpha, j} \bar{c}_{\alpha,k}
  = (1-z_j \bar{z_k})(\Gamma_{\alpha})_{jk} + (1-w_j\bar{w_k})
  (\Delta_{\alpha})_{jk}.
\end{equation}
The set of positive semi-definite matrices (of a fixed size) with
diagonal entries bounded by some constant is compact (their operator
norms are bounded by their traces which are uniformly bounded).  The
diagonal entries of $\Gamma_{\alpha}$ and $\Delta_{\alpha}$ are
bounded independently of $\alpha$ (e.g. it is not hard to prove
\[
\frac{1}{1-|z_j|^2} \geq (\Gamma_{\alpha})_{jj}
\]
for $j=1,\dots, N$) and therefore we may choose a subsequence so that
$\Gamma_{\alpha}$ converges to some positive semi-definite matrix
$\Gamma$ and $\Delta_{\alpha}$ converges to some positive
semi-definite matrix $\Delta$.  Therefore, if we take the limit as
$\alpha \to \infty$ in equation \eqref{alphadecomp} we have proved
\[
1-c_j \bar{c_k} = (1-z_j \bar{z_k})\Gamma_{jk} + (1-w_j\bar{w_k})
\Delta_{jk},
\]
which proves necessity in general.
\end{proof} 

\begin{question} \label{pickquestion} Can the uniqueness in Theorem
  \ref{mainthm} be carried over in some way to the above theorem?
\end{question}

Solutions to extremal Pick problems in two variables (those solvable
with a function of norm one but no less) are not unique as they are in
one variable, so we are necessarily vague in our question.

\section{Questions} \label{questions} We have already asked three
questions: Questions \ref{gcdquestion}, \ref{trigquestion}, and
Question \ref{pickquestion}.  Here are two others.  One of the most
fundamental questions to come out of our research is the following:
\begin{question} When is a rational function $p/q$ in
  $L^2(\mathbb{T}^2)$?  
\end{question}

Here we may as well assume $p, q \in \mathbb{C}[z,w]$ are relatively
prime but we are otherwise not imposing any conditions on their zero
sets.  If we impose restrictions, we can ask a more concrete question.

Suppose $q \in \mathbb{C}[z,w]$ has degree $(n,m)$, no zeros on the
bidisk, and finitely many zeros on $\mathbb{T}^2$ and suppose $p \in
\mathbb{C}[z,w]$ has degree $\leq (n-1,m-1)$. If $p/q \in
L^2(\mathbb{T}^2)$, then the sums of squares decomposition (as in
Theorem \ref{rearrangethm}) tells us that there is a constant $c$ such
that
\begin{equation} \label{estimate}
|q(z,w)|^2 - |\refl{q}(z,w)|^2 \geq c (1-|z|^2)(1-|w|^2)|p(z,w)|^2
\end{equation}
for $(z,w) \in \mathbb{D}^2$, since $p$ will be in $\gkboxsm_\mu$ for
the Bernstein-Szeg\H{o} measure $\mu$ associated to $q$.

\begin{question} Is the converse true? Does the estimate
  \eqref{estimate} imply $p/q \in L^2(\mathbb{T}^2)$?
\end{question}

\section*{Notational Index and Conventions}
In this section we index where various notations and terms are defined
in the paper. We also list our notational conventions.

\begin{tabular}{|l|l|} \hline
Notation & Location \\ \hline
$\refl{q}(z,w)$ & Def \ref{def:reflection} \\ \hline
$\bLam_n(z)$ $\bLam_m(w)$ & Eq \eqref{lambda} \\ \hline
$\mathbb{C}^N[z], \mathbb{C}^N[z,w]$ & Notation \ref{mainthmnotation}
\\ \hline
toral & Def \ref{def:toral} \\ \hline
atoral & Def \ref{def:atoral} \\ \hline
distinguished variety & Def \ref{def:distvar} \\ \hline
$d\sigma=d\sigma(z,w)$ & Eq \eqref{Lebesgue} \\ \hline
``degree $(n,m)$'' & Def \ref{def:degree} \\ \hline
$\hat{q}(j,k)$ & Eq \eqref{def:hat} \\ \hline
$\gkbox, \gkboxr, \gkboxu, \gkboxsm, \gkboxll, \gkboxur$ & Notation
\ref{boxnotation} \\ \hline
$\ip{f}{g}_\mu$ & Eq \eqref{innerprod} \\ \hline
$KV$, $K\gkbox_\mu$, etc. & Notation \ref{kernelnotation} \\ \hline
$w\gkboxsm_\mu, z\gkboxsm_\mu, \gkboxrperpdn_\mu, \gkboxllperp_\mu$,
etc. & Notation \ref{complementnotation} \\ \hline
$\mathcal{I}_\mu$ & Eq \eqref{def:ideal} \\ \hline
``$\mathbb{T}^2$-symmetric'' & Def \ref{t2symmetric} \\ \hline
$L_{(Z,W)}$ & Eq \eqref{def:L} \\ \hline
$\pi_1, \pi_2$ & Eq \eqref{def:pis} \\ \hline
$Z_q$ & Eq \eqref{def:Zq} \\ \hline
\end{tabular}

Notational conventions:

\begin{tabular}{|l|l|} \hline
$n,m$ & fixed positive integers (see Remark \ref{fixnm}) \\
$p,q$ & elements of $\mathbb{C}[z,w]$ \\ 
$\mathbf{E}, \mathbf{F},
\mathbf{G}, \mathbf{A}, \mathbf{B}, \mathbf{Q}$ & vector polynomials
\\ 
$E, F, A, B, Q$ & matrix polynomials in one variable \\ 
$\ip{\cdot}{\cdot}$ with no subscript & inner product on
$\mathbb{C}^N$ ($N$ determined from context) \\
$L^2(\mathbb{T}^2)$ & $L^2$ on the torus with respect to Lebesgue
measure \\ 
$L^2(\mu), L^2(\rho)$ & $L^2$ on the torus with respect to the measure
$\mu$ or $\rho$ \\ 
$H^2(\mathbb{T}), H^2(\mathbb{T}^2)$ & classical Hardy space on
$\mathbb{T}$ or $\mathbb{T}^2$ \\
$\Phi, \Psi$ & one variable matrix valued inner functions \\ \hline
\end{tabular}

\bibliography{nozonbidisk}

\end{document}